\documentclass[11pt]{amsart}     

\usepackage{amsmath,amssymb,amsfonts,amsthm} 
\usepackage{enumerate}

\newtheorem{defin}{{\bf Definition}}[section]
\newtheorem{thm}[defin]{{\bf Theorem}}

\newtheorem{lem}[defin]{{\bf Lemma}}
\newtheorem{rem}[defin]{{\bf Remark}}

\newtheorem{cor}[defin]{{\bf Corollary}}

\newcommand{\tr}{{\mathrm {tr}}}
\newcommand{\coll}{{\mathrm {coll}}}
\newcommand{\lift}{{\mathrm {lift}}}
\newcommand{\recipe}{{\mathrm {recipe}}}
\newcommand{\age}{{\mathrm {Age}}}

\newcommand{\N}{\mathbb{N}}
\newcommand{\funct}[2]{#1 \longrightarrow #2}
\newcommand{\m}[1]{\textbf{#1}}
\newcommand{\md}[1]{\textbf{#1}'}
\newcommand{\Cy}{\textbf{S}(2)}
\newcommand{\Cz}{\textbf{S}(3)}

\newcommand{\arrows}[3]{\longrightarrow {#1}^{#2}_{#3}}
\newcommand{\restrict}[2]{#1\mathbin{\upharpoonright} #2}

\begin{document}

\title[Ramsey expansions of homogeneous directed Graphs]{Ramsey precompact expansions of homogeneous directed graphs}

\author[J.Jasi\'{n}ski]{Jakub Jasi\'{n}ski}
\address{University of Calgary, Department of Mathematics and Statistics, Calgary, Alberta, Canada T2N 1N4} 
\email{jjasinsk@ucalgary.ca} 

\author[C.Laflamme] {Claude Laflamme*}
\thanks{*Supported by NSERC of Canada Grant \# 690404}
\address{University of Calgary, Department of Mathematics and Statistics, Calgary, Alberta, Canada T2N 1N4} 
\email {laflamme@ucalgary.ca} 

\author [L.Nguyen Van Th\'e]{Lionel Nguyen Van Th\'e**}
\thanks{**Supported by A.N.R. project Grupoloco (ANR-11-JS01-0008)}
\address{Aix-Marseille Universit\'e, CNRS, Centrale Marseille, I2M, UMR
7373, 13453 Marseille, France}
\email{lionel@latp.univ-mrs.fr}

\author[R.Woodrow] {Robert Woodrow}
\address{University of Calgary, Department of Mathematics and Statistics, Calgary, Alberta, Canada T2N 1N4} 
\email {woodrow@ucalgary.ca} 

\begin{abstract} In 2005, Kechris, Pestov and Todor\v{c}evi\'{c} provided a
powerful tool to compute an invariant of topological groups known as
the universal minimal flow, immediately leading to an explicit
representation of this invariant in many concrete cases. More
recently, the framework was generalized allowing for further
applications, and the purpose of this paper is to apply these new
methods in the context of homogeneous directed graphs.

In this paper, we show that the age of any
homogeneous directed graph allows a \emph{Ramsey precompact
expansion}. Moreover, we verify the relative expansion properties
and consequently describe the respective universal minimal flows.
\end{abstract}

\date{}
\maketitle

\section{Introduction}

The article \cite{key-KPT} by Kechris, Pestov and Todor\v{c}evi\'{c}
established a fundamental correspondence between structural Ramsey
theory and topological dynamics. As an immediate consequence, it
triggered a renewed interest for structural Ramsey theory, and also
motivated a more detailed investigation of the connection between
combinatorics and Polish group actions.  More precisely,
\cite{key-KPT} provided an extremely powerful tool to compute an
invariant known as the \emph{universal minimal flow}, and immediately
led to an explicit representation of this invariant in many concrete
cases, including the case of most homogeneous (simple, loopless)
graphs. However, in some particular situations including some
homogeneous directed graphs, that framework did not allow to perform
the computation directly.

Recently, the {\em ordering property} technique was generalized in
\cite{key-NVT} to the notion of {\em precompact expansion}, opening
the door for further applications. This was in particular used to
compute the universal minimal flows of the automorphism groups of the
homogeneous circular directed graphs $\Cy$ and $\Cz$. The purpose of
this paper is to apply these new methods in the context of homogeneous
directed graphs; following the classification by Cherlin in
\cite{key-C}, we show that in each case its age allows a Ramsey
precompact expansion. In this paper, it will be shown that in
addition, this expansion can be used to compute the universal minimal
flow of the corresponding automorphism group.

The underlying strategy we will use in order to construct Ramsey
expansions can be described as follows: starting from a homogeneous
directed graph, if the new age is not Ramsey by simply adding any
well-chosen linear ordering, then further expand using finitely
many additional relations.

\bigskip

In an effort to make the current paper relatively self-contained, we
review the main results and notation of \cite{key-KPT} and
\cite{key-NVT} relevant to the current work.  In what follows, $\N$
denotes the set $\{ 0, 1, 2, \ldots\}$ of natural numbers, and for a
natural number $m$, $[m]$ will denote the set $\{ 0, \ldots,
m-1\}$. We will assume that the reader is familiar with the concepts
of first order logic, first order structures, Fra\"iss\'e theory (cf
\cite{key-KPT}, section 2), reducts and expansions (cf \cite{key-KPT},
section 5). If $L$ is a first order signature and $\m A$ and $\m B$
are $L$-structures, we will write $\m A\leq \m B$ when $\m A$ embeds
in $\m B$, $\m A \subset \m B$ when $\m A$ is a substructure of $\m
B$, and $\m A \cong \m B$ when $\m A$ and $\m B$ are isomorphic.  If
$C$ is a subset of the universe of $\m A$ which supports a
substructure of $\m A$, we will write $\restrict{\m A}{C}$ for the
corresponding substructure.

A \emph{Fra\"iss\'e class} in a countable first order relational
language $L$ will be a countable class of finite $L$-structures of
arbitrarily large finite sizes, satisfying the hereditarity, joint
embedding and amalgamation property, and a \emph{Fra\"iss\'e
structure} (or \emph{Fra\"iss\'e limit}) in $L$ will be a countable,
locally finite, homogeneous $L$-structure. In general a relational
structure is called \emph{homogeneous} if any isomorphism between
finite substructures can be extended to an automorphism of the entire
structure. In \cite{key-KPT}, two combinatorial properties of classes
of finite structures have a considerable importance, those are the
\emph{Ramsey property} and the \emph{ordering property}. In order to
define the Ramsey property, let $k, l\in \N$, and $\m A, \m B, \m C$
be $L$-structures. The set of all \emph{copies} of $\m A$ in $\m B$ is
written using the binomial notation 
\[ \binom{\m B}{\m A} = \{ \md A \subset \m B : \md A \cong \m A\}.\] 
We use the standard arrow partition symbol 
\[\m C \arrows{(\m B)}{\m A}{k, l}\] 
to mean that for every map $c: \funct{\binom{\m C}{\m A}}{[k]}$,
thought as a $k$-coloring of the copies of $\m A$ in $\m C$, there is
a copy $\md B \in \binom{\m C}{\m B}$ such that $c$ takes at most
$l$-many values on $\binom{\md B}{\m A}$. When $l=1$, this is written
\[ \m C \arrows{(\m B)}{\m A}{k}.\] A class $\mathcal{K}$ of finite
$L$-structures is then said to have the \emph{Ramsey property} when
the following statement holds: 
\[ \forall k \in \N \quad \forall \m A, \m B \in \mathcal{K} \quad \exists \m C \in \mathcal{K} \quad \m C
\arrows{(\m B)}{\m A}{k}. \] 
When $\mathcal{K} = \age (\m F)$, where $\m F$ is a Fra\"iss\'e
structure, this is equivalent, via a compactness argument to: \[
\forall k \in \N \quad \forall \m A, \m B \in \mathcal{K} \quad \m F
\arrows{(\m B)}{\m A}{k}.\]

As for the ordering property used in \cite{key-KPT}, assume that $<$
is a binary relation symbol not contained in $L$, and that
$L^{*}=L\cup\{<\}$. For $\mathcal{K}$, a Fra\"iss\'e class in
$L$, we write $\mathcal{K}^{*}$ for an \emph{order expansion} of
$\mathcal K$ in $L^{*}$. This means that all elements of
$\mathcal{K}^{*}$ are of the form $\m A ^{*} = (\m A, <^{\m A})$,
where $\m A \in \mathcal{K}$ and $<^{\m A}$ is a linear ordering on
the universe $A$ of $\m A$; thus $\m A$ is then the \emph{reduct} of
$\m A^{*}$ to $L$ and is also denoted $\restrict{\m A^{*}}{L}$, or
even $\restrict{\m A^{*}}{\mathcal{K}}$, and that, conversely, any $\m
A \in \mathcal K$ admits a linear ordering $<^{\m A}$ so that $(\m A,
<^{\m A}) \in \age({\mathcal K}^{*})$. Then, $\mathcal K ^{*}$ is said
to satisfy the \emph{ordering property} relative to $\mathcal K$ if,
for every $ \m A \in \mathcal{K}$, there exists $ \m B \in
\mathcal{K}$ such that \[ \forall \m A^{*}, \m B^{*} \in\mathcal{K}
^{*} \quad (\restrict{\m A^{*}}{L} = \m A \ \ \wedge \ \ \restrict{\m
B^{*}}{L} = \m B) \Rightarrow \m A^{*} \leq \m B^{*}. \] The ordering
property is precisely the technique that was generalized in
\cite{key-NVT} to \emph{precompact (relational) expansions}. For such
expansions, we do not require $L^{*}=L\cup \{<\}$, but only $L^{*} =
L\cup \{ R_{i}:i\in I\}$, where $I$ is countable, and every symbol
$R_{i}$ is relational and not in $L$. An expansion $\m F^{*}$ of $\m
F$ is then called \emph{precompact} when any $\m A \in \age(\m F)$
only has finitely many expansions in $\age(\m F^{*})$.

In the spirit as above, we say that $\mathcal{K}^{*}$ has the
\emph{expansion property} relative to $\mathcal{K}$ if and only if for
every $\m A \in \mathcal{K}$, there is $\m B \in \mathcal{K}$ such
that 
\[ \forall \m A^{*}, \m B^{*} \in\mathcal{K} ^{*} \quad
(\restrict{\m A^{*}}{L} = \m A \ \ \wedge \ \ \restrict{\m B^{*}}{L} =
\m B) \Rightarrow \m A^{*} \leq \m B^{*}. \]
The importance of the new technique is demonstrated in the following
result, generalizing the results of \cite{key-KPT} where the ordering
property had been used:

\begin{thm} \cite{key-NVT} \label{thm:EP}
Let $\m F$ be a Fra\" iss\'e structure, and $\m F^{*}=(\m F, \vec R^{*})$ a (not
necessarily Fra\"iss\'e) precompact relational expansion of $\m
F$. Then the  following are equivalent:

\begin{enumerate} 
\item[i)] The logic action of $\mathrm{Aut}(\mathbf{F})$ on the space
$\overline{\mathrm{Aut}(\mathbf{F})\cdot \vec R^{*}}$ is
minimal. \item[ii)] $\age(\m F^{*})$ has the expansion property
relative to $\age(\m F)$. 
\end{enumerate} 
\end{thm}

\begin{thm} \cite{key-NVT} \label{thm:UMF}
Let $\m F$ be a Fra\" iss\'e structure, $\m F^{*}=(\m F, \vec R^{*})$ be a
Fra\"iss\'e precompact relational expansion of $\m F$, and further assume that
$\age(\m F^{*})$ consists of rigid elements. Then the following are
equivalent:

\begin{enumerate}
\item[i)] The logic action of $\mathrm{Aut}(\mathbf{F})$ on the space $\overline{\mathrm{Aut}(\mathbf{F})\cdot \vec R^{*}}$ is its  universal minimal flow. 
\item[ii)] The class $\age(\m F^{*})$ has the Ramsey property as well as the expansion property relative to $\age(\m F)$. 
\end{enumerate}
\end{thm}

\noindent Thus, given $\m F$ a Fra\" iss\'e structure, our general
objective is to find an expansion $\m F^{*}$ of $\m F$ that is
precompact, has the expansion property,  its age is Ramsey and consists of rigid members.  Note that if the age of $F^{*}$ consists of rigid members and satisfies the Ramsey property and the joint embedding property, then it also has the amalgamation property (see \cite{key-KPT}).  In
this paper, we complete this program for all homogeneous directed
graphs, moreover we do so using a well chosen and minimal number of
new relations. All arguments are very much standard, except those used
for the case of the semigeneric directed graph in Section
\ref{section:semi}.

\textbf{Acknowledgements}: This project was finished while both of us,
Claude Laflamme and Lionel Nguyen Van Th\'e, were attending the thematic
program Universality and Homogeneity at the Hausdorff Research Institute
for Mathematics in Bonn. We would therefore like to acknowledge the
support of the Hausdorff Research Institute for Mathematics, and thank the
organizers Alexander Kechris, Katrin Tent and Anatoly Vershik for having
made this stay possible.  We would also like to thank Miodrag Soki\'{c} for
his numerous suggestions.

\section{Homogeneous Directed Graphs}

\label{section:homogeneous}

A \emph{graph} $G$ is a pair $(V(G),E(G))$ where $V(G)$ is the set of
\emph{vertices} and $E(G)$ is the set of (undirected) edges.  A simple
(or undirected) graph is one where $E$ is a symmetric (and
irreflexive), and a \emph{directed graph} is one where $E$ is
asymmetric (and irreflexive). A \emph{tournament} is a complete
directed graph, that is every pair of distinct vertices supports
exactly one directed edge. A \emph{hypergraph} is a generalization of
a graph where edges are non-empty subsets, and we call it a
\emph{k-uniform hypergraph} if all these edges have the same size $k$.
A \emph{(2)-set system} in the sense of \cite{key-NR} is a linearly
ordered 2-uniform hypergraph where edges also carry a label.  The following
Ramsey result from \cite{key-NR} will play a central role in this
paper.

\begin{thm} \cite{key-NR}\label{thm:ramseysetsystems}
Given two finite set systems $A$ and \textbf{$B$} and $r\in\N$, there is a set system $C$ such that
\[ C\rightarrow(B)_{r}^{A}. \] 
\end{thm}

For example a linearly ordered tournament can be converted
into a set system as follows: if $(x,y)$ is a directed edge, then
form the hyperedge $\{x,y\}$ and label it with $0$ if $x<y$, $1$
otherwise. Vice-versa, a set system can be converted into a
directed graph by following the reverse procedure. 

By identifying graphs as relational structures with universe $V(G)$
and single relation $E(G)$, it is thus natural to consider
\emph{homogeneous} graphs. To complete our program, we will follow the
classification of homogeneous directed graphs as described in
\cite{key-C}. Concretely, it means that we will construct a precompact
Ramsey expansion for each of the directed graphs in Cherlin's list,
which we repeat here for convenience (more detail about each object
can be found in the corresponding section, or in \cite{key-C},
p.74-75):

\begin{enumerate}

\item[1.] $I_{n}$. 

\item[2.] $C_{3}, \mathbb Q, \Cy, T^{\omega}$. 

\item[3.] $\m T[I_{n}]$ (Section \ref{section:TIn} and \ref{section:TIomega}), $I_{n}[\m T]$ 
(Section \ref{section:InT} and \ref{section:IomegaT}), where $\m T$ is a homogeneous tournament. 

\item[4.] $\hat{\m T}$ for $\m T = I_{1}, C_{3}, \mathbb Q, T^{\omega}$ (Section \ref{section:TIn}). 

\item[5.] Complete $n$-partite (Section \ref{section:complete}). 

\item[6.] Semigeneric (Section \ref{section:semi}).

\item[7.] $\Cz$. 

\item[8.] $\mathcal P$. 

\item[9.] $\mathcal P (3)$ (section \ref{section:p3}).

\item[10.] $\Gamma_{n}$.

\item[11.] $\mathcal T$-generic. 

\end{enumerate}
  
The cases 1, 2, 7, 8, 10, 11 will not be treated here as Ramsey
precompact expansions with the expansion property are already known
for them. Precisely:

\begin{enumerate}

\item[1.] $I_{n}$ denotes the edgeless directed graph on $n$ vertices,
$n\leq\omega$. When $n<\omega$, $I_{n}$ is finite and the relevant
expansion $I_{n} ^{*}$ is obtained by adding $n$ unary relations, one
for each point in $I_{n}$. When $n=\omega$, $I_{\omega} ^{*}$ is
obtained by adding a linear ordering that makes it isomorphic to the
rationals (\cite{key-KPT}, based on the usual finite Ramsey theorem).

\item[2.] These are the homogeneous tournaments. $C_{3}$ denotes the $3$-cycle, 
and the corresponding expansion $C_{3} ^{*}$ is obtained in the same
way as it is obtained for $I_{n}$ by adding three unary
predicates. For convenience, we will also assume that $C_3 ^*$ is
equipped with a linear ordering $<^{C_3 ^*}$. This does not change the
automorphism group of $C_3 ^*$ but will be useful in Section 4.
$\mathbb Q$ denotes the rationals, seen as a directed
graph where $E(x,y)$ iff $x<y$. For this structure, simply take
$\mathbb Q^{*} = \mathbb Q$, and the Ramsey property follows again in
virtue of the usual finite Ramsey Theorem. $\Cy$ denotes the dense
local order; its points are the points lying at a rational angle on
the unit circle, with the edge relation defined by $E(x,y)$ iff the
angle from $x$ to $y$ is in the range $(0,
\pi)$. $\Cy ^{*}$ is constructed by adding two unary relations,
corresponding to the partition into left and right half (note that
there are no antipodal points at $\pi/2$ and $3\pi/2$).
Observe that there is a definable linear
ordering $<^{\Cy ^*}$ in that structure, obtained by reversing edges that
are between elements lying in different parts (see \cite{key-NVT},
Proposition 10, for more details). In what follows, it will be convenient
to include it as a relation of $\Cy^*$.
The underlying Ramsey theorem for this class is proved in \cite{key-LNS}, and the
universal minimal flow is described in \cite{key-NVT}. Finally,
$T^{\omega}$ denotes the generic tournament, and $T^{\omega *}$ is
obtained by adding a generic linear ordering; the corresponding Ramsey
theorem is a consequence of a general result proved independently by
Abramson-Harrington \cite{key-AH} and Ne\v set\v ril-R\"odl
\cite{key-NR1, key-NR2, key-NR}.

\item[7.] $\Cz$ is a variant of $\Cy$, where $E(x,y)$ iff the angle
from $x$ to $y$ is in the range $(0, 2\pi/3)$. $\Cz^{*}$ is
constructed by adding three unary relations, corresponding to the
partition into three arcs of same length and without extremity points
(same references as for $\Cy$). As for $\Cy^*$, there is a definable linear ordering
$<^{\Cz^*}$ in $\Cz^*$ (see \cite{key-NVT}, Proposition 12), which we will
consider as a relation in $\Cz^*$.

\item[8.] $\mathcal P$ denotes the generic partial order. Its
expansion is obtained by considering $\mathcal P ^{*}$, the
Fra\"iss\'e limit of the class of all finite partial orders that are
totally ordered by a linear extension (see \cite{key-PTW}, based on
unpublished results of Ne\v set\v ril-R\"odl).

\item[10.] $\Gamma_{n}$ denotes the generic directed graph where
$I_{n+1}$ does not embed. When $n=1$, this is nothing else than
$T^{\omega}$, for which the relevant expansion was already
described. When $n>1$, $\Gamma_{n} ^{*}$ is also obtained by adding a
generic linear ordering. The underlying Ramsey theorem is covered by
general results of Ne\v set\v ril-R\"odl thanks to the partite
construction: by switching edges and
non-edges, the class becomes the class of $K_n$-free graphs, which is
treated by \cite{key-NR}.

\item[11.] $\mathcal T$-generic, where $\mathcal T$ denotes a set of
finite tournaments, refers to the generic directed graph as the
Fra\"{\i}ss\'e limit of all the finite directed graphs which do not
embed any member of $\mathcal T $. As in case 10, the general results
by Ne\v set\v ril-R\"odl allow to prove that the relevant expansion is
obtained by adding a generic linear ordering.

\end{enumerate}

\section{A Ramsey Lemma for Products}

This section is devoted to an elementary but powerful finite
combinatorial principle.

\begin{defin}
Let $n\in\N^+$ and $0\leq a_{i}\leq b_{i}\in\mathbb{N}$
for $1\leq i\leq n$. If $B_{i}\subseteq\mathbb{N}$ are disjoint sets of size 
$|B_{i}|=b_{i}$ for $1\leq i\leq n$, then define
\[
{B_{1},\ldots,B_{n} \choose a_{1},\ldots,a_{n}}
\]
to consist of all sets of the form $A'=\bigcup_{i=1}^{n}A_{i}$, where
$A_{i}\subseteq B_{i}$ and $|A_{i}|=a_{i}$ for $1\leq i\leq n$.
\end{defin}

Using this notation, we have the following Ramsey lemma for products.

\begin{lem}[\cite{key-GRS}] \label{lem:blob}
Let $r,n\in\N^+$ and let $0\leq a_{i}\leq b_{i}\in\mathbb{N}$
for $1\leq i\leq n$. Then there exist disjoint $C_{i}\subseteq\mathbb{N}$
for $1\leq i\leq n$, such that given any colouring $\chi$ of 
\[
{C_{1},\ldots,C_{n} \choose a_{1},\ldots,a_{n}}
\]
with $r$ colours, there exists $B_{i}\subseteq C_{i}$ and $|B_{i}|=b_{i}$ so that $\chi$
is monochromatic on
\[
{B_{1},\ldots,B_{n} \choose a_{1},\ldots,a_{n}}.
\]
 We symbolize the statement above with
\[
(C_{1},\ldots,C_{n})\rightarrow(b_{1},\ldots,b_{n})_{r}^{(a_{1},\ldots,a_{n})}.
\]
If $b=b_{i}$ for all $i$, we set $N=\max\{|C_{i}|:1\leq i\leq n\}$
and write
\[
N\rightarrow(b)_r^{(a_{1},\ldots,a_{n})}.
\]
\end{lem}

There is an ordered version of the above lemma.  Consider the
class $\mathcal{OP}_{n}$ of structures of the form $\m A = (A,<^{\m A}, P_{1}^{\m A},\ldots,P_{n}^{\m A})$  where $<^{\m A}$ is a linear ordering on $A$, while $\{P_i\}_{1\leq i \leq n}$ is partition of $A$ into disjoint sets.   It
follows from the result in \cite{key-KPT} (see proof of Theorem 8.4) that
$\mathcal{OP}_{n}$ satisfies the Ramsey property. We will also have to consider sequences of
elements of $\mathcal{OP}_{n}$ and hence require the following Ramsey
result which follows from Theorem 2 for sequences of classes in
\cite{key-S1}.

\begin{lem}\label{lem:blob2}
Let $n,k,r\in\mathbb{N}$. Let $(\mathbf{A}_{1},\ldots,\mathbf{A}_{n})$
and $(\mathbf{B}_{1},\ldots,\mathbf{B}_{n})$ be sequences of elements
in $\mathcal{OP}_{k}$. Then there is a sequence
$(\mathbf{C}_{1},\ldots,\mathbf{C}_{n})$ of elements in
$\mathcal{OP}_{k}$ such that for any colouring $\chi$ of sequences
$(\mathbf{A}_{1}',\ldots,\mathbf{A}_{n}')$ such that
$\mathbf{A}'_{i}\in{\mathbf{C}_{i} \choose \mathbf{A}_{i}}$ with $r$
colours, there is a sequence
$(\mathbf{B}_{1}',\ldots,\mathbf{B}_{n}')$ such that
$\mathbf{B}'_{i}\in{\mathbf{C}_{i} \choose \mathbf{B}_{i}}$ and $\chi$
assumes only one colour on the set of sequences
$(\mathbf{A}{}_{1}'',\ldots,\mathbf{A}_{n}'')$ such that
$\mathbf{A}''_{i}\in{\mathbf{B}_{i}' \choose \mathbf{A}_{i}}$. \\
Furthermore, we can make the $C_i$'s disjoint and consequently extend
the orders to an ordering of $\bigcup_i C_i$, convex with respect to
$C_i$'s. 
\end{lem}

\section{$\mathbf{T}[I_{n}]$, $\hat{\mathbf{T}}$}
\label{section:TIn}

We are now ready to undertake our program starting with the above
homogeneous directed graphs. We start with a brief description of each of them. 

When $\m T$ is a homogeneous tournament and $n$ is a positive integer,
the structure $\mathbf{T}[I_{n}]$ is defined as the homogeneous
directed graph whose vertex set consists of pairs $(x,i)$ with $x\in
T$ and $i\in [n]$. A pair $((x,i),(y,j))$ is an edge of
$\mathbf{T}[I_{n}]$ if and only if $(x,y)$ is an edge of $\mathbf{T}$.


The structure $\hat{\mathbf{T}}$ is identical to $\mathbf{T}[I_{2}]$ except
that $((y,i),(x,j))$ is an edge of $\hat{\mathbf{T}}$ if and only if
$(x,y)$ is an edge of $\mathbf{T}$ and $i\neq j$.

We now turn to the description of the corresponding expansions $\m T
[I_{n}]^{*}$ and $\hat{\m T} ^{*}$. To construct
the structure $\m T [I_{n}]^{*}$, we start from $\m T^{*}$ as
described at the end of Section \ref{section:homogeneous} and use it
to naturally expand $\m T [I_{n}]$. Precisely, consider the language
$L_{\m T ^{*}}$ of $\m T ^{*}$. It contains the edge symbol $E$ from
$L_{\m T}$ and a distinguished relation symbol $<$, interpreted as a
linear ordering $<^{*}$ in $\m T^{*}$. For  $R$ is a relational symbol
in $L_{\m T ^{*}}\smallsetminus \{E,  <\}$, set $R^{\m T
[I_{n}]^{*}}(\vec u)$ iff $R^{\m T ^{*}}(\vec x)$, where $\vec u = ((x_1,i_1),\ldots,(x_n,i_n))$. As for $<$,
interpret it as the lexicographical ordering of $<^{*}$ and the usual
order on $[n]$. Finally, add to the resulting structure $n$ new unary
predicates $(L_{i})_{i\in [n]}$ defined by $L_{i}(x,j)$ iff $j=i$. The
structure $\hat{\m T} ^{*}$ is constructed similarly. 

It is clear that those expansions are precompact since we added only
finitely many predicates. We now show that their age has the Ramsey
property. In order to do this, it suffices to concentrate on the case
of $\m T [I_{n}]^{*}$, as the proof in the two other cases follows the
same scheme. Let $\mathbf{A}$ be a finite substructure of
$\mathbf{T}[I_{n}]^{*}$. Let $\varphi:A\rightarrow[n]$ be such that
$\varphi(x)=i$ if $L_{i}(x)$. Define the \emph{collapse} of $\m A$,
$\coll(\mathbf{A})$, to be the projection of $\mathbf{A}$ to
$\mathbf{T}$, i.e., $x$ is in the universe of $\coll(\m A)$ if and
only if $(x,i)\in A$ for some $i$. Construct the sequence
\[
\recipe(\mathbf{A})=\left(\{n_{1}^{1},\ldots,n_{k_{1}}^{1}\},\ldots,\{n_{1}^{l},\ldots,n_{k_{l}}^{l}\}\right)
\]
as follows: Enumerate $\coll(\mathbf{A})=x_{1}<\ldots<x_{l}$ consistently
with the ordering on \textbf{$\mathbf{T}^{*}$}.
Set $A_{x}=A\cap {x} \times [n]$, and let $\{n_{1}^{j},\ldots,n_{k_j}^{j}\}=\varphi''A_{x_{j}}$.

Conversely, given a sequence 
\[
\sigma=\left(\{n_{1}^{1},\ldots,n_{k_{1}}^{1}\},\ldots,\{n_{1}^{l},\ldots,n_{k_{l}}^{l}\}\right)
\]
where $n_{j}^{i} \in  [n]$  and given a finite substructure of $\mathbf{A}$
of $\mathbf{T}$ of size $l$, define $\lift(\mathbf{A},\sigma)$  as follows:
Enumerate $\mathbf{A}=x_{1}<\ldots<x_{l}$ consistently with the ordering
on \textbf{$\mathbf{T}^{*}$}. The universe of $\lift(\mathbf{A},\sigma)$
consists of all pairs $(x_{i},n_{j}^{i})$ with $1\leq j\leq k_{i}$.

Thus we have the following immediate Lemma.

\begin{lem}
\label{lem:collapse}Suppose $\mathbf{A}\cong\mathbf{A}'$ are finite
substructures of $\mathbf{T}[I_{n}]^{*}$.
Then
\[
\lift(\coll(\mathbf{A}'),\recipe(\mathbf{A}))=\mathbf{A}'.
\]
\end{lem}

\begin{thm}
\label{thm:TIn-and-T-hat} The class $\age(\mathbf{T}[I_n]^{*})$ satisfies the Ramsey
Property.
\end{thm}

\begin{proof}
Let $\mathbf{A}$, $\mathbf{B}$ be finite substructures of
$\mathbf{T}[I_{n}]^{*}$ and $r\in\mathbb{N}$. We need to produce a
finite substructure $\mathbf{C}\subseteq\mathbf{T}[I_{n}]^{*}$ such
that
\[
\mathbf{C}\rightarrow(\mathbf{B})_{r}^{\mathbf{A}}.
\]

Because $\age(\mathbf{T}^{*})$ satisfies the Ramsey Property, there exists
a finite substructure $\mathbf{C} _{0}\subseteq\mathbf{T}^{*}$ such that 
\[
\mathbf{C}_{0}\rightarrow(\coll(\mathbf{B}))_{r}^{\coll(\mathbf{A})}.
\]
Let $\m C$ be the substructure of $\mathbf{T}[I_{n}]^{*}$ obtained from $\m C_{0}$ following the same procedure that turned $\m T^{*}$ into $\mathbf{T}[I_{n}]^{*}$. We claim that $\m C$ is as required. Let $r\geq2$ and
\[
\chi:{\mathbf{C} \choose \mathbf{A}}\rightarrow [r]
\]
be a colouring. Define $\bar{\chi}:{\mathbf{C} _{0} \choose
\coll(\mathbf{A})}\rightarrow r$ by
$\bar{\chi}(\mathbf{A}')=\chi(\lift(\mathbf{A}',\recipe(\mathbf{A}))$
for any $\mathbf{A}'\in{\mathbf{C}_{0} \choose \coll(\mathbf{A})}$. Then there
exists $\mathbf{B}'\in{\mathbf{C}_{0} \choose \coll(\mathbf{B})}$ such
that $\bar{\chi}$ is constant on ${\mathbf{B}' \choose
\coll(\mathbf{A})}$.  Note that
$\mathbf{B}''=\lift(\mathbf{B}',\recipe(\mathbf{B}))\in{\mathbf{C} \choose
\mathbf{B}}$.

\noindent Now for $\mathbf{A}',\mathbf{A}''\in{\mathbf{B}'' \choose \mathbf{A}}$, we obtain by Lemma \ref{lem:collapse}:
\begin{align*}
\chi(\mathbf{A}') & =\chi(\lift(\coll(\mathbf{A}'),\recipe(\mathbf{A}))\\
 & =\bar{\chi}(\coll(\mathbf{A}'))\\
 & =\bar{\chi}(\coll(\mathbf{A}''))\\
 & =\chi(\lift(\coll(\mathbf{A}''),\recipe(\mathbf{A}))\\
 & =\chi(\mathbf{A}'').
\end{align*}
\end{proof}

Similarly we can step up the expansion property, and again we provide
the proof only for the case of  $\mathbf{T}[I_{n}]$ as the others are similar. 

\begin{thm}
\label{thm:exp1}If $\age(\mathbf{T}^{*})$ satisfies the expansion
property with respect to $\age(\mathbf{T})$, then
$\age(\mathbf{T}[I_{n}]^{*})$ satisfies the expansion property with
respect to $\age(\mathbf{T}[I_{n}])$ .\end{thm}

\begin{proof}
Consider $\mathbf{A}\in\age(\mathbf{T}[I_{n}])$. Then
$\bar{\mathbf{A}}=\coll(\mathbf{A}) \in \age(\mathbf{T})$ and by
assumption we can find a structure
$\bar{\mathbf{B}}\in\age(\mathbf{T})$ such that any $\mathbf{T}^*$
expansion of $\bar{\mathbf{A}}$ embeds in any $\mathbf{T}^*$ expansion
of $\bar{\mathbf{B}}$. We claim that $\bar{\mathbf{B}}[I_{n}]
\in\age(\mathbf{T}[I_{n}])$ is is the required structure.

This is because any $\mathbf{T}[I_{n}]^{*}$ expansion $\mathbf{A}^{*}$ of
$\mathbf{A}$ is of the form $\lift(\bar{\mathbf{A}}^{*},\sigma)$, for
some recipe $\sigma$ and some $\mathbf{T}^*$ expansion $\bar{\mathbf{A}}^{*}$ of
$\bar{\mathbf{A}}$. Similarly
 any $\mathbf{T}[I_{n}]^{*}$ expansion $\mathbf{B}^{*}$ of
$\bar{\mathbf{B}}[I_{n}]$ is of the form $\lift(\bar{\mathbf{B}}^{*},\alpha)$, for
some recipe $\alpha$ and some $\mathbf{T}^*$ expansion $\bar{\mathbf{B}}^{*}$ of
$\bar{\mathbf{B}}$.
Since $\bar{\mathbf{A}}^{*}$ embeds in $\bar{\mathbf{B}}^{*}$, it is easily seen that 
$\mathbf{A}^{*}$ embeds into $\mathbf{B}^{*}$.
\end{proof}

We state the expansion property for the other cases for reference.

\begin{thm}
If $\age(\mathbf{T}^{*})$ satisfies the expansion property with respect
to $\age(\mathbf{T})$, then $\age(\hat{\mathbf{T}}^*)$ satisfies the
expansion property with respect to $\age(\hat{\mathbf{T}})$.\end{thm}



\section{$\mathbf{T}[I_{\omega}]$}

\label{section:TIomega}

In this section, we treat the case of $\mathbf{T}[I_{\omega}]$
separately; the reason is that the unary predicable $L_i$ are not needed
here so extra care must be taken. We first deal with the case when $\m
T$ is an infinite homogeneous tournament, and the
structure $\mathbf{T}[I_{\omega}]$ is defined in the same way as
$\mathbf{T}[I_{n}]$ for $n$ finite: it is the homogeneous directed
graph whose vertex set consists of pairs $(x,i)$ with $x\in T$ and
$i\in \N$. A pair $((x,i),(y,j))$ is an edge of
$\mathbf{T}[I_{\omega}]$ if and only if $(x,y)$ is an edge of
$\mathbf{T}$.

The structure $\mathbf{T}[I_{\omega}]^{*}$ is also defined in the same
way as $\mathbf{T}[I_{n}]^{*}$ was from $\m T ^{*}$, except that the
linear ordering is the lexicographical product of the ordering $<^{*}$
of $\m T^{*}$ with an ordering $\prec$ on $\N$ that makes $(\N,
\prec)$ isomorphic to $\mathbb Q$, and that as mentioned  no unary predicate $L_{i}$
is added. Note that starting from a finite substructure $\bar{\m A}$
of $\m T ^{*}$ and from a positive integer $n$, the same procedure
yields a finite substructure $\bar{\m A}[I_{n}]^{*}$ of
$\mathbf{T}[I_{\omega}]^{*}$.

Suppose now $\mathbf{A}$ is a finite substructure of
$\mathbf{T}[I_{\omega}]^{*}$.  Define as before the collapse
$\coll(\mathbf{A}) \subseteq \mathbf{T}$ as the projection on its
first coordinates. Enumerating $\coll(\mathbf{A})=x_1<\ldots<x_k$
using the ordering on $\m T^{*}$, define the \emph{signature} of
$\mathbf{A}$ to be $\sigma = (a_1, \ldots, a_k)$ where 
$a_i= | A \cap \{(x_i,j): j \in \N \} |$.

We have the following:

\begin{lem}
\label{lem:infinite1}Let $\mathbf{A}$ be a finite
substructure of $\mathbf{T}[I_{\omega}]^{*}$ and $\mathbf{\bar{A}} =
\coll(\mathbf{A})$ (so that $\mathbf{\bar A} \subseteq
\mathbf{T}^{*}$). Then for any $r$ and any $n$, there is an
$N\in\mathbb{N}$ such that for any colouring
\[
\chi:{\mathbf{\bar{A}}[I_{N}]^{*} \choose \mathbf{A}}\rightarrow [r]
\]
 there exists $\mathbf{A}'\in{\mathbf{\bar{A}}[I_{N}]^{*} \choose \mathbf{\bar{A}}[I_{n}]^{*}}$
such that $\chi$ is monochromatic on ${\mathbf{A}' \choose \mathbf{A}}$.
\end{lem}

\begin{proof}
Let $\sigma=\left(a_1,\ldots,a_k \right)$ be the
signature of $\mathbf{A}$.  By Lemma \ref{lem:blob}, there exists $N$ such that
\[
N\rightarrow(n)_{r}^{(a_{1},\ldots,a_{k})}.
\]

The result follows from the fact that there is a one-to-one
relationship between ${\mathbf{\bar{A}}[I_{N}]^{*} \choose
\mathbf{A}}$ and ${N,\ldots,N \choose a_1,\ldots,a_k}$.
\end{proof}

Using the preceding Lemma, we can thus guarantee that the colour of
copies of $\mathbf{A}$ depends only on their collapse.

\begin{cor}
\label{cor:infinite1}Let $\mathbf{A}$ be a finite
substructure of $\mathbf{T}[I_{\omega}]^{*}$, and $\mathbf{\bar{C}}$
be a finite substructure of $\mathbf{T}^{*}$.  For any $r$ and any
$n\in\N$, there is an $N\in\mathbb{N}$ such that for any colouring
\[
\chi:{\mathbf{\bar{C}}[I_{N}]^{*} \choose \mathbf{A}}\rightarrow r
\]
there exists $\mathbf{C}'\in{\mathbf{\bar{C}}[I_{N}]^{*} \choose
\mathbf{\bar{C}}[I_{n}]^*}$ such that for any
$\mathbf{\bar{A}}\in
{\mathbf{\bar{C}} \choose \coll(\mathbf{A})}$, $\chi$ is monochromatic on 
${\mathbf{C'}\cap \mathbf{\bar{A}}[I_N]^*  \choose \mathbf{A}}$.
\end{cor}

\begin{proof}
First enumerate ${\mathbf{\bar{C}} \choose \coll(\mathbf{A})} =
\mathbf{\bar{A}_1}, \ldots, \mathbf{\bar{A}_\ell}$. Now let $N_0=n$,
and for $i<\ell$ inductively use Lemma \ref{lem:infinite1} to find
$N_{i+1}$ so that there exists
$\mathbf{A}'\in{\mathbf{\bar{A}_i}[I_{N_{i+1}}]^{*} \choose
\mathbf{\bar{A}_i}[I_{N_i}]^{*}}$ such that $\chi$ is monochromatic on
${\mathbf{A}' \choose \mathbf{A}}$.  Finally, set $N=N_\ell$.

Equivalently, if $\sigma=\left(a_1,\ldots,a_k \right)$ is the
signature of $\mathbf{A}$, then for $i<\ell$ choose $N_{i+1}$ so that
$N_{i+1} \rightarrow(N_i)_{r}^{(a_{1},\ldots,a_{k})}$ and let
$N=N_\ell$.

\end{proof}

We are now ready to  show that $\age(\mathbf{T}[I_{\omega}]^{*})$ satisfies the
Ramsey property, assuming $\mathbf{T}^{*}$ satisfies it in the first place. 

\begin{rem}
It came to attention of the authors that one can also obtain the following result by applying Theorem 1.3 from \cite{key-S4}.
\end{rem}

\begin{thm}
The class $\age(\mathbf{T}[I_{\omega}]^{*})$ satisfies the Ramsey property.
\end{thm}

\begin{proof}
We need to show that if $\mathbf{A}\subseteq\mathbf{B}$ are finite substructures of
$\mathbf{T}[I_{\omega}]^{*}$, then for any $r$, there is a finite
$\mathbf{C}\subseteq\mathbf{T}[I_{\omega}]^{*}$ such that
\[
\mathbf{C}\rightarrow(\mathbf{B})_{r}^{\mathbf{A}}.
\]
Without loss of generality we may suppose that
$\mathbf{B}=\bar{\mathbf{B}}[I_{n}]^{*}$ where $\bar{\mathbf{B}}$ is a
finite substructure of $\mathbf{T}^{*}$.

Since $\age(\mathbf{T}^{*})$ satisfies the Ramsey property, let
$\mathbf{\bar{C}}$ be a finite substructure of $\mathbf{T}^{*}$ such that
\[
\bar{\mathbf{C}}\rightarrow(\bar{\mathbf{B}})_{r}^{\coll(\mathbf{A})}.
\]

By Corollary \ref{cor:infinite1}, there is $N$ and
$\mathbf{C}'\in{\mathbf{\bar{C}}[I_{N}]^{*} \choose
\mathbf{\bar{C}}[I_{n}]^*}$ such that $\chi$ is monochromatic on 
all copies of $\mathbf{A}$ in $\mathbf{C'}$ having the same collapse. 

Thus define $\bar{\chi}:{\mathbf{\bar{C}} \choose
\coll(\mathbf{A})}\rightarrow [r]$ by
$\bar{\chi}(\bar{\mathbf{A}})=\chi(\mathbf{A}')$ for some (any)
$\mathbf{A}'\in{\mathbf{C}' \choose \mathbf{A}}$ such that
$\coll(\mathbf{A'})=\bar{\mathbf{A}}$.

This is well defined by construction, and now by assumption there is
$\tilde{\mathbf{B}}\in{\bar{\mathbf{C}} \choose
\bar{\mathbf{B}}}$ such that $\bar{\chi}$ is monochromatic on
${\tilde{\mathbf{B}} \choose \coll(\mathbf{A})}$.  Finally set 
\[
\mathbf{B}'=\tilde{\mathbf{B}}[I_{N}]^{*}\cap\mathbf{C}'.
\]

Since $\mathbf{C}'\cong\bar{\mathbf{C}}[I_{n}]^{*}$, we have
$\mathbf{B}'\cong\bar{\mathbf{B}}[I_{n}]^{*}=\mathbf{B}$.

Now for $\mathbf{A}',\mathbf{A}''\in{\mathbf{B}' \choose \mathbf{A}}$,
we have $\coll(\mathbf{A}'),\coll(\mathbf{A}'')\in{\tilde{\mathbf{B}}
\choose \coll(\mathbf{A})}$, and therefore:
\[
\chi(\mathbf{A}')=\bar{\chi}(\coll(\mathbf{A}'))=\bar{\chi}(\coll(\mathbf{A}''))=\chi(\mathbf{A}'').
\]
 
This completes the proof.
\end{proof}

\begin{thm}
If $\age(\mathbf{T}^{*})$ satisfies the expansion property with
respect to $\age(\mathbf{T})$, then $\age(\mathbf{T}[I_{\omega}]^{*})$
satisfies the expansion property with respect to
$\age(\mathbf{T}[I_{\omega}])$.
\end{thm}

\begin{proof}
Suppose $\mathbf{A}\in\age(\mathbf{T}[I_{\omega}])$. It embeds in
$\bar{\mathbf{A}}[I_{n}]$ for some $n$, where $\bar{\mathbf{A}} = \coll(\mathbf{A})\in\age(\mathbf{T})$.
Find $\bar{\mathbf{B}}\in\age(\mathbf{T})$ such that any expansion of
$\bar{\mathbf{A}}$ embeds in any expansion of $\bar{\mathbf{B}}$.
When expanding $\mathbf{A}$, each part is linearly ordered. Therefore,
the convexly ordered expansion $\mathbf{A}^{*}$ must embed in any
expansion $\bar{\mathbf{B}}[I_{n}]^{*}$ (recall we omitted the relations
$L_{i}$).\end{proof}

If $\m T$ is finite, then only the case $\m T = C_3$ needs attention
since the case $\m T=I_1$ is trivial.  So let $\m T=C_3$ and expand
$\mathbf{T}[I_{\omega}]$ to $\mathbf{T}[I_{\omega}]^{*}$ as follows:
We can assume the universe is  $T = [3] $, and let $P_0, \ldots, P_2$ be
unary predicates and $<$ be a binary predicate.  The expansion
$\mathbf{T}[I_{\omega}]^{*}$ is obtained by interpreting $P_i = \{
(i,j): (i,j)\in \N \}$  for $i \in [3]$, and $<^*$ as a  dense  linear order on each $P_i$  
such that $P_0 <^* P_1 <^* P_2$.

\begin{thm}\label{FinTOmega-RP}
$\age(C_3[I_{\omega}]^{*})$ satisfies the Ramsey property.
\end{thm}

\begin{proof} 
A finite structure $\m A ^* \in \age(C_3[I_{\omega}] ^*)$ can be
viewed as an element in ${\N,\N,\N \choose a_1,a_2,a_3}$.
 Therefore, the Ramsey property follows from
Lemma \ref{lem:blob}.
\end{proof}

Similarly the expansion property is immediate.

\begin{thm}\label{FinTOmega-exp}
$\age(C_3[I_{\omega}]^{*})$ satisfies
the expansion property with respect to $\age(C_3[I_{\omega}])$.
\end{thm}

And this also completes our program verification in this case.  

\section{$I_{n}[\m T]$}

\label{section:InT}

We again treat the case with $n$ finite separately. For $\mathbf{T}$
be a countable homogeneous tournament, the structure $I_{n}[\m T]$ is
then the disjoint union of $n$ copies of $\m T$. We define the
structure $I_{n}[\m T]^{*}$ as the disjoint union of $n$ copies of $\m
T^{*}$, for which each copy of $\m T^{*}$ corresponds to a unary
predicate, and which is equipped with a convex linear order that
orders the parts in a given specific pattern. More formally, the
universe of $I_{n}[\m T]^{*}$ is made of all pairs $(i,x)$ with $i\in
[n]$ and $x\in T^{*}$. When $R$ is a relational symbol in $L_{\m T
^{*}}\smallsetminus \{ <\}$, set $R^{\m T [I_{n}]^{*}}(\vec u)$ iff $R^{\m T ^{*}}(\vec
x)$, where $\vec u = ((k,x_1),\ldots,(k,x_n))$, for some $k$. As for $<$, interpret it as the lexicographical ordering of the
usual order on $[n]$ and $<^{*}$ as per $\m T^*$. Finally, add to the resulting
structure $n$ new unary predicates $(P_{i} ^{I_{n}[\m T]^{*}})_{i\in
[n]}$ defined by $P_{i}^{I_{n}[\m T]^{*}}(j,x)$ iff $j=i$. Note that
therefore $P_{0}^{I_{n}[\m T]^{*}}<^{I_{n}[\m T]^{*}}\ldots<^{I_{n}[\m
T]^{*}}P_{n-1} ^{I_{n}[\m T]^{*}}$. Note also that starting from a
finite substructure $\mathbf{\bar A}$ of $\m T^{*}$, the same
procedure yields a finite substructure $I_{n}[\mathbf{\bar A}]^{*}$ of
$I_{n}[\m T]^{*}$.


\begin{rem}
The short proof of Theorem \ref{thm:finite1} is essentially the same as the standard proof of Lemma \ref{lem:blob} and Theorem 2 from \cite{key-S1}.  We retain it here to keep the paper as self-contained as possible.
\end{rem}

\begin{thm} \label{thm:finite1}
The class $\age(I_n[\mathbf{T}]^{*})$ satisfies the Ramsey Property.
\end{thm}

\begin{proof}
Let $r$ be a positive integer and $\mathbf{A}\subseteq\mathbf{B}\subseteq
I_{n}[\m T]^{*}$, we need to find $\mathbf{C}\subseteq I_{n}[\m T]^{*}$
such that
\[
\mathbf{C}\rightarrow(\mathbf{B})_{r}^{\mathbf{A}}.
\]

We proceed by induction on $n$, the base case $n=1$ simply being the
Ramsey property of $\age(\mathbf{T}^{*})$.

For the induction step, suppose the desired property is true for $n-1$. First set
$A_{i}=P_{i}^{I_n[T]^{*}}\cap A$ and $B_{i}=P_{i}^{I_n[T]^{*}}\cap B$. Note that
\[
\mathbf{A}^{0}:=\mathbf{A}_{0}\sqcup\ldots\sqcup\mathbf{A}_{n-2}\subseteq I_{n-1}[\m T]^{*}
\]
 and
\[
\mathbf{B}^{0}:=\mathbf{B}_{0}\sqcup\ldots\sqcup\mathbf{B}_{n-2}\subseteq I_{n-1}[\m T]^{*}.
\]

Let $\mathbf{C}^{1}$ be such that
\[
\mathbf{C}^{1}\rightarrow(\mathbf{B}_{n-1})_{r^{\left|{\mathbf{B}^{0} \choose \mathbf{A}^{0}}\right|}}^{\mathbf{A}_{n-1}}
\]

Let $\mathbf{C}^{0}$ be such that 
\[
\mathbf{C}^{0}\rightarrow(\mathbf{B}^{0})_{r^{\left|{\mathbf{C}^{1} \choose \mathbf{A}_{n-1}}\right|}}^{\mathbf{A}^{0}}.
\]
Then  $\mathbf{C}=\mathbf{C}^{0}\sqcup\mathbf{C}^{1}$ is as required. Indeed let 
\[
\chi:{\mathbf{C} \choose \mathbf{A}}\rightarrow [r]
\]
 be any colouring. Define 
\[
\chi_{0}:{\mathbf{C}^{0} \choose \mathbf{A}^{0}}\rightarrow \left[ r^{\left|{\mathbf{C}^{1} \choose \mathbf{A}_{n-1}}\right|}\right]
\]
by setting
\[
\chi_{0}(\mathbf{A}')=\left\{ (\chi(\mathbf{A}'\sqcup\mathbf{A}''),\mathbf{A}''):\mathbf{A}''\in{\mathbf{C}^{1} \choose \mathbf{A}_{n-1}}\right\} 
\]
for $\mathbf{A}'\in {\mathbf{C}^{0} \choose \mathbf{A}^{0}}$.

Subsequently, using a $\chi_{0}$-monochromatic $\mathbf{B}_{0}\in{\mathbf{C}^{0} \choose \mathbf{B}^{0}}$,
further define 
\[
\chi_{1}:{\mathbf{C}^{1} \choose \mathbf{A}_{n-1}}\rightarrow\left[r\right]
\]
by 
\[
\chi_{1}(\mathbf{A}'')=\chi(\mathbf{A}'\sqcup\mathbf{A}'')
\]
where $\mathbf{A}''$ is any member of ${\mathbf{C}^{1} \choose \mathbf{A}_{n-1}}$, for some $\mathbf{A}'\in{\mathbf{B}_{0} \choose \mathbf{A}^{0}}$.
This is independent of the choice of $\mathbf{A}'$ since $\mathbf{B}_{0}$
is $\chi_{0}$-monochromatic.

Therefore, there is a $\chi_{1}$-monochromatic
$\mathbf{B}^{1}\in{\mathbf{C}^{1} \choose \mathbf{B}_{n-1}}$, and thus
$\mathbf{B}_{0}\sqcup\mathbf{B}_{1}$ is $\chi$-monochromatic as desired.
\end{proof}

%
%
%
%
%
%
\begin{thm}
If $\age(\mathbf{T}^{*})$ satisfies the expansion property with respect
to $\age(\mathbf{T})$, then $\age(I_{n}[\mathbf{T}]^{*})$ satisfies
the expansion property with respect to $\age(I_{n}[\mathbf{T}])$.
\end{thm}

\begin{proof}
Suppose $\mathbf{A}\in\age(I_{n}[\mathbf{T}])$. Then $\mathbf{A}$
has the universe of the form $\{0\}\times A_{0}\sqcup\ldots\sqcup\{n-1\}\times A_{n-1}$
with $\mathbf{A}_{i}\in\age(\mathbf{T})$ or $A_{i}=\emptyset$. For
every $i$, if $A_{i}\neq\emptyset$, find $\mathbf{B}_{i}\in\age(\mathbf{T})$
such that any $\m T^*$ expansion of $\mathbf{A}_{i}$ embeds in any $\m T^*$ expansion
of $\mathbf{B}_{i}$. If $A_{i}=\emptyset$, set $B_{i}=\emptyset$.
We can assume that the universes $B_{i}$ are pairwise disjoint.
Any $I_{n}[\mathbf{T}]^*$ expansion $\mathbf{A}^{*}$ of $\mathbf{A}$ has a universe
of the form $\{0\}\times A_{f(0)}^{*}\sqcup\ldots\sqcup\{n-1\}\times A_{f(n-1)}^{*}$
for some bijection $f:[n]\rightarrow[n]$, where each $\mathbf{A}_{f(i)}^{*}$
is a $\m T^*$  expansion of $\mathbf{A}_{i}$ and each $P_{i}^{\mathbf{A}^{*}}=A_{f(i)}^{*}$.
Set $\mathbf{B}$ to be the convexly ordered structure in $\age(I_{n}[\mathbf{T}]^{*})$
with universe 
\[
B=\{0\}\times\left(\bigcup_{i\in [n]}B_{i}\right)\sqcup\ldots\sqcup\{n-1\}\times\left(\bigcup_{i\in [n]}B_{i}\right)
\]
where the substructures with universes $\{k\}\times B_{i}$ are isomorphic
to $\mathbf{B}_{i}$, for all $i$ and $k$. Therefore, the  $I_{n}[\mathbf{T}]^*$ expansion
$\mathbf{A}^{*}$ must embed in any  $I_{n}[\mathbf{T}]^*$ expansion $\mathbf{B}^{*}$.\end{proof}

\section{$I_{\omega}[\m T]$}

\label{section:IomegaT}

For $\mathbf{T}$ a countable homogeneous tournament, the structure
$I_{\omega}[\m T]$ is the disjoint union of countably many copies of
$\m T$. Formally, its universe is $\N\times T$, and $((i,x)(j,y))$ is
an edge in $I_{\omega}[\m T]$ iff $i=j$ and $(x,y)$ is an edge in $\m
T$. The structure $I_{\omega}[\m T]^{*}$ is defined in the same way as
$I_{n}[\m T]^{*}$ was from $\m T ^{*}$, except that the linear
ordering is the lexicographical product of a dense linear order
$\prec$ on $\N$ with the ordering $<^{*}$ of $\m T^{*}$. Here, no
unary relation $P_{i}$ is added. Note also that starting from a
finite substructure $\bar{\m A}$ of $\m T ^{*}$ and from a positive
integer $n$, the same procedure yields a finite substructure
$I_{n}[\bar{\m A}]^{*}$ of $I_{\omega}[\mathbf{T}]^{*}$.

For $i\in \N$, define the \emph{part} $X_{i}$ to be the set of all
pairs $(i,x)$ for $x \in \m T$. For any $\mathbf{A}\subseteq
I_{\omega}[\m T]^{*}$, define the \emph{collapse} $\coll(\mathbf{A})$
to be the set $\{i:A \cap X_{i}\neq\emptyset\}$.  If $\mathbf{\bar
A}\subseteq\mathbf{T}^{*}$ and $\sigma$ is a finite subset of
$\omega$, then we define $\sigma[\mathbf{\bar A}]^{*}$ to be the
substructure of $I_{\omega}[\m T]^{*}$ consisting of $(i,x)$ with
$x\in \mathbf{\bar A}$ and $i\in\sigma$. Define ${\mathbf{B} \choose
\mathbf{A}}_{\tau}$ to be the subset of ${\mathbf{B} \choose
\mathbf{A}}$ consisting of those copies of $\mathbf{A}$ whose collapse
is precisely $\tau$.

The following is a direct consequence of Theorem \ref{thm:finite1}.

\begin{lem}
\label{lem:infinite2-1}Let $\mathbf{A}$ and $I_{m}[\mathbf{\bar B}]^{*}$
be a finite substructures of $I_{\omega}[\mathbf{T}]^{*}$ and let
$\tau$ be a subset of $[m]$ of size $|\coll(\mathbf{A})|$. For any $r\geq2$, there
is a finite $\mathbf{\bar C}\subseteq\mathbf{T}^{*}$ such that for any colouring
\[
\chi:{I_{m}[\mathbf{\bar C}]^{*} \choose \mathbf{A}}\rightarrow [r]
\]
 there exists $\mathbf{B}'\in{I_{m}[\mathbf{\bar C}]^{*} \choose I_{m}[\mathbf{\bar B}]^{*}}$
such that $\left|\chi'' {\mathbf{B}' \choose \mathbf{A}}_{\tau}\right|=1$. \\
We symbolize this by
\[
I_{m}[\mathbf{\bar C}]^{*}\rightarrow(I_{m}[\mathbf{\bar B}]^{*})_{r}^{\mathbf{A},\tau}.
\]
\end{lem}

An iteration of the above argument yields the following. 

\begin{cor}
\label{cor:infinite2-2}Let $\mathbf{A}$ and $I_{m}[\mathbf{\bar B}]^{*}$
be a finite substructures of $I_{\omega}[\mathbf{T}]^{*}$. For any
$r\geq2$, there is a finite $\mathbf{\bar C}\subseteq\mathbf{T}^{*}$ such
that for any colouring 
\[
\chi:{I_{m}[\mathbf{\bar C}]^{*} \choose \mathbf{A}}\rightarrow [r]
\]
 there exists $\mathbf{B}'\in{I_{m}[\mathbf{\bar C}]^{*} \choose I_{m}[\mathbf{\bar B}]^{*}}$
such that $\left|\chi''{\mathbf{B}' \choose \mathbf{A}}_{\tau}\right|=1$
for any $\tau\in{m \choose \coll(\mathbf{A})}$. \\
We symbolize this by
\[
I_{m}[\mathbf{\bar C}]^{*}\rightarrow(I_{m}[\mathbf{\bar B}]^{*})_{r}^{\mathbf{A},\coll}
\]
\end{cor}

\begin{proof}
First enumerate the set ${m\choose \coll(\mathbf{A})}$ of
$|\coll(\mathbf A)|$-subsets of $[m]$ as \newline $\tau_{1},
\ldots, \tau_\ell$. Let $\mathbf{\bar C}_1$ be such that
\[ I_{m}[\mathbf{\bar C}_1]^{*}\rightarrow(I_{m}[\mathbf{\bar B}]^{*})_{r}^{\mathbf{A},\tau_{1}} \]
and for $i<\ell$ inductively use Lemma \ref{lem:infinite2-1} to find
$\mathbf{\bar C}_{i+1}$ such that
\[ I_{m}[\mathbf{\bar C}_{i+1}]^{*}\rightarrow(I_{m}[\mathbf{\bar C}_{i}]^{*})_{r}^{\mathbf{A},\tau_{i+1}} \]
Finally, set $\mathbf{\bar C}=\mathbf{\bar C}_\ell$.
\end{proof}

\begin{rem}
As before, it came to attention of the authors that, one can also obtain the following result, by applying Theorem 1.4 from \cite{key-S4}.
\end{rem}

\begin{thm}
\label{thm:infinite2-3} The class 
$\mathrm{Age} ( I_{\omega}[\m T]^{*})$ satisfies the Ramsey Property.
\end{thm}

\begin{proof}
Let $r\in\N$ be positive and let $\mathbf{A}\subseteq
\mathbf{B}\subseteq I_{\omega}[\m T]^{*}$. We need to find
$\mathbf{C}\subseteq I_{\omega}[\m T]^{*}$ such that
\[
\mathbf{C}\rightarrow(\mathbf{B})_{r}^{\mathbf{A}}.
\]
Without loss of generality, we may suppose that
$\mathbf{B}=I_{m}[\mathbf{\bar{B}}]^{*}$. If $|\coll(\mathbf{A})|=k$,
then by Ramsey's theorem there is $N$ such that
\[
N\rightarrow(m)_{r}^{k}.
\]
By Corollary \ref{cor:infinite2-2}, there is a finite $\mathbf{\bar C}\subseteq\mathbf{T}^{*}$
such that 
\[
I_{N}[\mathbf{\bar C}]^{*}\rightarrow(I_{N}[\mathbf{\bar{B}}]^{*})_{r}^{\mathbf{A},\coll}.
\] 
Let $\chi:{I_{N}[\mathbf{\bar{B}}]^{*} \choose \mathbf{A}}\rightarrow
[r]$ be any colouring. By the above, we can find
$\mathbf{B}'\in{I_{N}[\mathbf{\bar C}]^{*} \choose
I_{N}[\mathbf{\bar{B}}]^{*}}$ such that for any $\tau\in{m \choose
\coll(\mathbf{A})}$, we have $\left|\chi''{\mathbf{B}' \choose
\mathbf{A}}_{\tau}\right|=1$.  ({*})

Define $\bar{\chi}:{N \choose \coll(\mathbf{A})}\rightarrow [r]$ as
follows: If $\tau\in{N \choose \coll(\mathbf{A})}$, set $\bar{\chi}(\tau)=\chi(\mathbf{A}')$
for some (any) $\mathbf{A}'\in{\mathbf{B}' \choose \mathbf{A}}_{\tau}$.
The colouring $\bar{\chi}$ is well defined by ({*}). Now, we can
find $\sigma\in{N \choose m}$ such that $\bar{\chi}$
is monochromatic on ${\sigma \choose \coll(\mathbf{A})}$. Set 
\[
\mathbf{B}''=\{(i,x)\in I_{\omega}[\mathbf{T}]^{*}:i\in\sigma\}\cap\mathbf{B}'.
\]
 Since $\mathbf{B}'\cong I_{N}[\mathbf{\bar B}]^{*}$, we have $\mathbf{B}''\cong I_{m}[\mathbf{\bar{B}}]^{*}$.
Let $\mathbf{A}',\mathbf{A}''\in{\mathbf{B}'' \choose \mathbf{A}}$.
Thus, $\coll(\mathbf{A}'),\coll(\mathbf{A}'')\in{\sigma \choose |\coll(\mathbf{A})|}$.
Therefore, 
\[
\chi(\mathbf{A}')=\bar{\chi}(\coll(\mathbf{A}'))=\bar{\chi}(\coll(\mathbf{A}''))=\chi(\mathbf{A}'')
\]
which completes the proof.

\end{proof}

\begin{thm}
If $\age(\mathbf{T}^{*})$ satisfies the expansion property with respect
to $\age(\mathbf{T})$, then $\age(I_{\omega}[\mathbf{T}]^{*})$ satisfies
the expansion property with respect to $\age(I_{\omega}[\mathbf{T}])$.
\end{thm}

\begin{proof}
Suppose $\mathbf{A}\in\age(I_{\omega}[\mathbf{T}])$. Without loss of
generality, we can assume that $\mathbf{A}$ has the universe of the
form $\{0\}\times A_{0}\sqcup\ldots\sqcup\{n-1\}\times A_{n-1}$ with
$\mathbf{A}_{i}\in\age(\mathbf{T})$. For every $i$, find
$\mathbf{B}_{i}\in\age(\mathbf{T})$ such that any expansion of
$\mathbf{A}_{i}$ embeds in any expansion of $\mathbf{B}_{i}$, and we
can assume that the universes $B_{i}$ are pairwise disjoint. Further,
we can assume that any expansion $\mathbf{A}^{*}$ of $\mathbf{A}$ has
the universe of the form $\{0\}\times
A_{f(0)}^{*}\sqcup\ldots\sqcup\{n-1\}\times A_{f(n-1)}^{*}$ for some
bijection $f:[n]\rightarrow[n]$, where each $\mathbf{A}_{f(i)}^{*}$ is an
expansion of $\mathbf{A}_{i}$ and $A_{f(i)}^{*}<A_{f(j)}^{*}$ for
$i<j$. Set $\mathbf{B}$ to be the convexly ordered structure in
$\age(I_{\omega}[\mathbf{T}]^{*})$ with universe
\[
B=\{0\}\times\left(\bigcup_{i \in [n]}B_{i}\right)\sqcup\ldots\sqcup\{n-1\}\times\left(\bigcup_{i \in [n]}B_{i}\right)
\]
where the substructures with universes $\{k\}\times B_{i}$ are isomorphic
to $\mathbf{B}_{i}$, for all $i$ and $k$. Moreover, we ensure that
$(i,a)<(j,b)$ for $i<j$ and all $a,b\in B$. Therefore, the expansion
$\mathbf{A}^{*}$ must embed in any expansion $\mathbf{B}^{*}$.\end{proof}

If $\m T$ is finite, then again only the case $\m T= C_3$ requires
attention.  Expand $I_{\omega}[\mathbf{T}]$ to
$I_{\omega}[\mathbf{T}]^{*}$ as follows: 

We can assume the universe is $T = [3] $, and let $L_0, \ldots, L_2$
be unary predicates and $<$ be a new binary predicate.  The expansion
$\mathbf{T}[I_{\omega}]^{*}$ is obtained by interpreting $L_i = \{
(j,i): j \in \N \}$ for $i \in [3]$, and $<^*$ as a dense linear order
on each $L_i$ such that $L_0 <^* L_1 <^* L_2$.  We can immediately
verify the following.

\begin{lem}\label{lem:IomegaFinT}
Let $\m A, \m A' \subseteq I_{\omega}[\mathbf{T}]^{*}$.  Then $\m A$
and $\m A'$ are isomorphic if and only if there is an order preserving
bijection $f : \coll(\m A) \rightarrow \coll( \m A')$ such that for
every $i$, $\{j : (i,j)\in A\} = \{j : (f(i),j) \in A'\}$.
\end{lem}

From this we immediately obtain the following using the usual Ramsey Theorem.

\begin{thm}\label{IomegaFinT-RP}
The class $\age(I_{\omega}[C_3]^{*})$ satisfies the Ramsey property.
\end{thm}

And similarly for the expansion property.

\begin{thm}\label{IomegaFinT-exp}
The class $\age(I_{\omega}[C_3]^{*})$ satisfies the expansion property with respect to  $\age(I_{\omega}[C_3])$.
\end{thm}

\begin{proof}
Let $\m A \in \age(I_\omega[C_3])$, then any  expansion $\m A ^*$ of 
$\m A$ will embed in any expansion of
\[ \{(i,j): i\in \coll( \m A), j \in [3]\}. \]
\end{proof}

\section{Complete $n$-partite directed graphs}
\label{section:complete}

For $n\leq \omega$, by a complete $n$-partite directed graph we mean a
directed graph such that the non-edge relation $\perp$ is an
equivalence relation with at most $n$ many classes. For each $n$, the
class of all such graphs is a Fra\"iss\'e class. The corresponding
limit is denoted $n*I_{\omega}$, and consists of $n$ many disjoint
copies of $I_{\omega}$ between which edges are distributed in a
complete and random way. For $n\in \N$ we construct the expansion
$n*I_{\omega}^{*}$ by adding a convex linear ordering together with
$n$ many unary relations $P_{i}^{n*I_{\omega}^{*}}$ corresponding to
the parts of $n*I_{\omega}$ and so that
$P_{0}^{n*I_{\omega}^{*}}<^{n*I_{\omega}^{*}}\ldots<^{n*I_{\omega}^{*}}P_{n-1}
^{n*I_{\omega}^{*}}$. For $n=\omega$, construct
$\omega*I_{\omega}^{*}$ by simply adding to $\omega*I_{\omega}$ a
convex linear ordering that orders each part in a dense linear order,
as well as the set of parts in a dense linear order; note that we do
not add any unary predicates in this case. Those expansions are
precompact, and we show that their ages have the Ramsey property and
show the appropriate expansion properties.

If $\m B$ is a complete $n$-partite directed graph, say that a
subgraph is \newline \emph{transversal} when it has exactly one point in each
$\perp$ equivalence class of $\m B$.


\begin{rem}It came to the attention of the authors that there is a simpler proof of the lemmas below, based on the result in \cite{key-S3}:  One takes the output of the result of \cite{key-S3} and prunes it, arriving at an element of the age of $n*I_{\omega}$.  Though, we include the direct proof based on the partite construction, we will revisit the result of \cite{key-S3} in section \ref{section:p3}. 
\end{rem}

\begin{lem}\label{lem:partite}
Let $n, r\in\N$ be positive and suppose $\m B$ is a finite
substructure of $n*I_{\omega}^{*}$ with exactly $n$ parts. Suppose $\m
A$ is an $n$-partite directed subgraph of $\m B$ which is
transversal. Then there exists a finite $n$-partite graph $\m C$ such
that
\[
\m C\rightarrow(\m B)_{r}^{\m A}.
\]
\end{lem}
\begin{proof}
This is identical to the proof of the partite lemma in \cite{key-NR}.
\end{proof}

Suppose now that $\m A$ is a directed subgraph of $\m B$, and that the
parts of $\m B$ are written $Y_{i}$ with $i\in [n]$. Define the
\emph{trace} of $\m A$ (in $\m B$) to be
\[
\tr(\m A)=\{i:A\cap Y_{i}\neq\emptyset\}.
\]
For $x\in B$, define the trace of $x$ to be $tr(x)=\tr(\{x\})$, thus $x\in Y_{\tr(x)}$.

We begin with the following special case of the Ramsey property. 

\begin{lem}\label{lem:partite2} 
Let $n,r\in\N$ with $r>0$, suppose $\m A \leq \m B$ are finite
substructure of $n*I_{\omega}^{*}$ with exactly $n$ parts each.  Then there
exists a finite $\m C \subseteq n*I_{\omega}^{*}$ such that
\[
\m C\rightarrow(\m B)_{r}^{\m A}.
\]
\end{lem}

\begin{proof}
What follows is a modification of the partite construction in
\cite{key-NR}.  Let $\m A$ and $\m B$ be as above, and denote by
$Y_{i}$, with $i<n$, the parts of $\m B$. Set $a=|A|$, $b=|B|$,
$a_{i}=|A\cap Y_{i}|$, and $b_{i}=|B\cap Y_{i}|$ for $i<n$. By Lemma
\ref{lem:blob}, there exist disjoint $C_{i}\subseteq\mathbb{N}$ with
$i<n$, such that

\[
(C_{0},\ldots,C_{n-1})\rightarrow(b_{0},\ldots,b_{n-1})_{r}^{(a_{0},\ldots,a_{n-1})}.
\]
We can assume that $C_{0}=\{1,\ldots,|C_{0}|\}$ and in general, \[C_{i}=\{\max(C_{i-1})+1,\ldots,\max(C_{i-1})+|C_i|\}.\]  Set $p=\max(C_{n-1})$. Enumerate
\[
{C_{0},\ldots,C_{n-1} \choose a_{0},\ldots,a_{n-1}}=\{N_{1},\ldots,N_{\alpha}\},
\]
 and 
\[
{C_{0},\ldots,C_{n-1} \choose b_{0},\ldots,b_{n-1}}=\{M_{1},\ldots,M_{\beta}\}.
\]
Now for every $j\in[\beta]$, $M_j$ is a disjoint union of $A_k^j$'s such that $A_k^j = M_j \cap C_k$, for every $k<n$. Construct a substructure $\m Q^{0}$ of $n*I_{\omega}^{*}$ as follows:
the vertex set $Q^{0}$ consists of all pairs $(i,j)$ such that $i\in M_{j}$. For each $j\in[\beta]$, reproduce a copy ${\m B}_j$ of $\m B$ with universe $\{(i,j): i\in M_{j}\}$, so that $P_{k}^{{\m B}_j}=B_j\cap C_k$ (for every $k<n$) and $(i,j)<^{{\m B}_j}(i',j)$ if and only if $i<i'$.   To complete our set of edges, for every $i<i'$ and $j<j'$, if $i\in A_k^j$ and $i'\in A_{k'}^{j'}$ where $k<k'$, then we add $((i,j),(i',j'))$. For every $k\in [n]$, $P_{k}^{{\m Q}^{0}}$ is the disjoint union of $P_{k}^{{\m B}_j}$'s.  Moreover, we set each non-structural part to be
\[
\mathcal{M}_i = \{(i,j):j\in[\beta]\}
\]
for $i\in [p]$.
Finally, add a linear ordering (convex with respect to both partitions) that extends the linear orderings that already exist on each copy of $\m B$: If $j<j'$, then set $(i,j)<^{{\m Q}^{0}} (i,j')$.   If $i<i'$, $j\neq j'$, then set  $(i,j)<^{{\m Q}^{0}} (i',j')$. The resulting structure is $\m Q ^{0}$. 

Next, we construct a sequence $\m Q^{1},\ldots$, $\m Q^{\alpha}$
inductively using lemma \ref{lem:partite2}, as in \cite{key-NR}. To
construct $\m Q ^{i}$ from $\m Q ^{i-1}$, we first consider the
directed subgraph of $\m Q ^{i-1}$ whose trace is $N^{i-1}$ with
respect to the non-structural partition, call it $\m
H^{i-1}$. Applying lemma \ref{lem:partite2}, we get an $a$-partite $\m
G^{i-1}$ such that
\[
\m G^{i-1}\rightarrow(\m H^{i-1})_{r}^{\m A}.
\]
Subsequently, we expand each copy of $\m H^{i-1}$ to a copy of $\m
Q^{i-1}$, and add a convex (with respect to the structural partition)
linear ordering that extends all the linear orderings that already
exist on each copy of $\m Q^{i-1}$. We call the result $\m
Q^{i}$. Following this procedure, the directed graph induced by the
structure $\m Q^{\alpha}$ is not complete with respect to the
structural partition, but it can be made complete by adding new
edges. Let $\m C$ be any structure obtained that way. We claim that
$\m C$ is as required.

Let $\chi:{\m C \choose \m A}\rightarrow [r]$ be any colouring. Then
there will be $\m Q'\in{\m C \choose \m Q ^{0}}$ such that the
colouring of each copy of $\m A$ depends only on its trace with
respect to the non-structural partition.  The conclusion follows from
the construction of $\m Q^{0}$ now using routine arguments.
\end{proof}

We are now ready to prove the general result.

\begin{thm}
\label{thm:partite}
The age of $n*I_{\omega}^{*}$ has the Ramsey property for each $n\in
\N$.
\end{thm}

\begin{proof} 
Let $k\leq n, r\in\mathbb{N}$ with $r>0$, $\m B$ a finite substructure
of $n*I_{\omega}^{*}$ with exactly $n$ parts, and $\m A$ is a
substructure of $\m B$ with $k$ parts. Then we will show that there exists a finite $\m
C \subseteq n*I_{\omega}^{*}$ with $n$ parts such that
\[
\m C\rightarrow(\m B)_{r}^{\m A}.
\]
Note that because of the unary relations that are interpreted as the
parts of $n*I_{\omega}^{*}$, all copies of $\m A$ in $\m B$ have same
trace $\tau$. Let $\mathbf{\bar{B}}$ be the substructure of $\m B$
induced by the points $v \in B$ with $\tr(v)\in\tau$.  Apply Lemma
\ref{lem:partite2} to find $\mathbf{\bar{C}}$ such that
$\mathbf{\bar{C}} \rightarrow (\mathbf{\bar{B}})^{\m A}_{r}$.  
$\sigma$ be the trace of $\m B$.  Now extend each copy of
$\mathbf{\bar{B}}$ in $\mathbf{\bar{C}}$ to a copy of
$\mathbf{\bar{B}}$, via amalgamation.  Set $\mathbf{C}$ to be the result of this
extension with appropriate edges and linear ordering added.
\end{proof}

%
%
%
%
%

This result allows us to show the  same is true for the infinite partite structure.

\begin{thm}
\label{thm:partite2}
The age of $\omega*I_{\omega}^{*}$ has the Ramsey property. 
\end{thm}

\begin{proof}

Let $k\leq n, r\in\mathbb{N}$ with $r>0$, $\m B$ a finite substructure
of $\omega*I_{\omega}^{*}$ with $n$ parts, and $\m A$  a
substructure of $\m B$ with $k$ parts.

Expand for the moment $\omega*I_{\omega}^{*}$ with unary predicates
corresponding to its parts, calling the resulting structure
$\omega*I_{\omega}^{**}$. Using Ramsey's theorem, we can find $N$ such
that $N\rightarrow(n)^{k}_{r}$, and we enumerate the $n$-subsets of $[N]$ as
$\tau_1, \ldots, \tau_\ell$. For each $i\leq l$ consider a copy $\m
B_{i}$ of $\m B$ with trace $\tau_{i}$ in $\omega*I_{\omega}^{*}$ so
that the $\m B_{i}$'s are pairwise disjoint.  Set $\mathbf{\bar{B}}$
to be the substructure of $\omega*I_{\omega}^{*}$ supported by the
union of these $\m B_{i}$'s.

Note that all copies of $\m A$ with fixed trace in
$\omega*I_{\omega}^{*}$ induce isomorphic structures in
$\omega*I_{\omega}^{**}$. Using Theorem \ref{thm:partite}, we can
find a finite substructure $\m C$ of $\omega*I_{\omega}^{*}$ such that
every $r$-colouring $\chi$ of ${\m C \choose \m A}$, there exists a
copy $\mathbf{\bar{B}}' \in {\m C \choose
\mathbf{\bar{B}}}$ such that $\chi$ restricted to ${\mathbf{\bar{B}}' \choose \m A}$
depends only on the trace.

To complete the proof , define $\bar{\chi} : {N \choose k}
\rightarrow [r]$ by setting $\bar{\chi}(\sigma)$ to be the colour of any
copy of $\m A$ in $\mathbf{\bar{B}}'$ with trace $\sigma$; this is well defined
by construction.  By the choice of $N$, there is a
$\bar{\chi}$-monochromatic $\tau_{i_0}$, and by construction, there is
$\m B'\in{ \mathbf{\bar{B}}' \choose \m B}$ with trace $\tau_{i_0}$, which completes
the proof.
\end{proof}


%
%
%

We now more to the expansion property, and note that the proof of the
theorem below follows the template found in the proof of Lemma 5.2 in
\cite{key-N} and the proofs of the expansion properties found in
\cite{key-S2}.

\begin{thm}
\label{thm:exp7}The age of $\omega*I_{\omega}^{*}$ satisfies the
expansion property with respect to the age of $\omega*I_{\omega}$.
\end{thm}

\begin{proof}
Let $\mathbf{A}\in\age(\omega*I_{\omega})$, and fix an arbitrary a convex linear ordering
$<$ on $A$. We will produce a structure  $\mathbf{B}\in\age(\omega*I_{\omega})$ such that 
 $\mathbf{A}^*$ embeds in any expansion of $\mathbf{B}$.

Consider the following structures:
\begin{itemize}
\item $\mathbf{E}_{1}$ is the linearly ordered directed graph with universe $\{a,b\}$ and
the directed edge $(a,b)$ such that $a<b$.
\item $\mathbf{E}_{2}$ is the linearly ordered directed graph with universe $\{a,b\}$ and 
no edges such that $a<b$.
\end{itemize}
Set $<_\mathrm{rev}$ to be the same as $<$ but with the ordering on
each part (i.e., the ordering on non-edges) reversed.  Set
$>_\mathrm{rev}$ to be the same as $>$ but with the ordering on each
part reversed.  \\ Now find $\mathbf{A}_1\in\age(\omega*I_{\omega}^{*})$
so that $(\mathbf{A},<)$, $(\mathbf{A},>)$,
$(\mathbf{A},<_\mathrm{rev})$ and $(\mathbf{A},>_\mathrm{rev})$ embed
in $\mathbf{A}_1$, and similarly find
$\mathbf{A}_2\in\age(\omega*I_{\omega}^{*})$ so that $\mathbf{A}_1$,
$\mathbf{E}_{1}$ and $\mathbf{E}_{2}$ all embed in $\mathbf{A}_2$. \\
Since $\age(\omega*I_{\omega}^{*})$ satisfies the Ramsey property, we
can find $\mathbf{C}_{i}\in\age(\omega*I_{\omega}^{*})$ so that
$\mathbf{C}_{1}\rightarrow(\mathbf{A}_2)_{2}^{\mathbf{E}_{1}}$ and 
$\mathbf{C}_{2}\rightarrow(\mathbf{C}_{1})_{2}^{\mathbf{E}_{2}}$. \\
Let $<'$ be the ordering on $\m C _2$, and $\prec$ be any convex linear
ordering on $C_{2}$, and consider the following colouring on ordered
pairs in the universe of $\mathbf{C}_{2}$:\\
\[
\chi((u,v))=\begin{cases}
0 & \text{if the two orders $<'$ and $\prec$ agree on $\{u,v\}$.} \\
1 & \text{otherwise}
\end{cases}
\]
This colouring induces colourings $\chi_i$ on the copies of
$\mathbf{E}_{i}$ in $\mathbf{C}_{2}$ for each $i$. By the above Ramsey
construction, there is $\mathbf{B}\in{\mathbf{C}_{2} \choose
\mathbf{A}_2}$ which is monochromatic with respect to each $\chi_{i}$.  In other
words, we have the following possibilities:
\begin{enumerate}
\item The ordering $\prec$ on $\mathbf{B}$ agrees with $<'$. 
\item The ordering $\prec$ on $\mathbf{}$ agrees with $<'$ 
on edges and goes in the opposite direction on non-edges.
\item The ordering $\prec$ on $\mathbf{B}$ agrees with $<'$ 
on non-edges and goes in the opposite direction on edges.
\item The ordering $\prec$ on $\mathbf{B}$ goes in the opposite direction.
\end{enumerate}
Now, since $\mathbf{A}_1$ embeds in $\mathbf{B}$, the structure
$(\mathbf{A},<)$ must embed in $\mathbf{B}$ with the order $<'$
replaced by $\prec$ and the proof is complete.
\end{proof}

\begin{thm}
\label{thm:exp8} The age of $n*I_{\omega}^{*}$ satisfies the expansion
property with respect to the age of $n*I_{\omega}$.
\end{thm}

\begin{proof}
Let $\mathbf{A}\in\age(n*I_{\omega})$, fix a sequence of relations
$P_i ^{\m A}$ corresponding to each part of $\m A$, and fix first a convex
linear ordering $<$ on $A$ that agrees with these relations. 

For $1\leq i \leq n$, define $\mathbf{E}^{i}\in\age(n*I_{\omega}^{*})$
to be the structure with universe $\{a,b\}$ with no edges and such that
$a<b$ and  $P_i^{\m E^{i}} = \{a,b\}$.

Enumerate all convex orderings on $A$ that agree with $P_i^{\m A}$'s:
$<_1, \ldots, <_o$. Find $\mathbf{B}\in\age(n*I_{\omega}^{*})$ so that
each $(\mathbf{A},P_1^{\m A},\ldots,P_n^{\m A},<_i)$ embeds in
$\mathbf{B}$, for $1\leq i \leq o$. Then find
$\mathbf{D}\in\age(n*I_{\omega}^{*})$ so that $\mathbf{B}$ and all
$E^{i}$'s embed in $\mathbf{D}$. Since $\age(n*I_{\omega}^{*})$
satisfies the Ramsey property, we can find
$\mathbf{C}_{i}\in\age(n*I_{\omega}^{*})$ so that
$\mathbf{C}_{1}\rightarrow(\mathbf{D})_{2}^{\mathbf{E}^1}$,
$\mathbf{C}_{i+1}\rightarrow(\mathbf{C}_{i})_{2}^{\mathbf{E}^{i+1}}$.
Let $<'$ be the ordering on $\m C _n$, and let $\prec$ be any convex
linear ordering on $C_{n}$ that agrees with $P_i^{\m C _n}$'s.

As before consider the following colourings on the copies of $\mathbf{E}^{i}$
in $\mathbf{C}_{n}$:
\[
\chi_{i}(\m E')=\begin{cases}
0 & \text{if the two orders $<'$ and $\prec$ agree on  $E'$.} \\
1 & \text{otherwise}
\end{cases}
\]
By the Ramsey property we can thus find $\mathbf{D}'\in{\mathbf{C}_{n}
\choose \mathbf{D}}$ which is monochromatic with respect to
$\chi_{1}$, $\ldots$, $\chi_{n}$.  In other words, $\prec$ either
agrees with $<'$ or goes in the opposite direction when restricted to
any part.  Furthermore, $\prec$ already agrees with $<'$ on edges,
because we have assumed that $\prec$'s convexity agrees with the
relations $P_i ^{\m C _n}$.  Now, since $\mathbf{B}$ embeds in
$\mathbf{D'}$, the structure $(\mathbf{A},P_1^{\m A},\ldots,P_n^{\m
A},<)$ must embed in $\mathbf{D}'$ with the order $<'$ replaced by
$\prec$.

Let $\m C' $ the reduct of $\m C_n$ resulting from removing $<'$.
Given a bijection $f : [n] \rightarrow [n]$, define $\m C _f$ to be
$\m C '$ except that $P_i^{\m C _f} = P_{f(i)}^{\m C '}$.  Find a
structure $\m C ''$ so that every $\m C _f$ embeds into it.  Set $\m C
_0 \in \age (n*I_\omega)$ to be the reduct of $\m C ''$. For any
$P_1^{\m C _0},\ldots,P_n^{\m C _0}$ and any convex linear order
$<''$, we can find a convex linear order $<'''$ so that \[(\m C ',
P_1^{\m C '},\ldots,P_n^{\m C '}, <''')\] embeds in \[(\m C _0,
P_1^{\m C _0},\ldots,P_n^{\m C _0}, <''),\] by construction.
Therefore, \[(\mathbf{A},P_1^{\m A},\ldots,P_n^{\m A},<)\] embeds in
any expansion \[(\m C _0, P_1^{\m C _0},\ldots,P_n^{\m C _0}, <''),\]
given any $P_1^{\m C_0},\ldots,P_n^{\m C_0}$, and any convex linear
order $<''$.

Finally, to find the witness of the expansion property for $\m A$,
take the joint embedding of all $\m C _0$'s, for all sequences
$P_1^{\m A},\ldots,P_n^{\m A}$ and all convex linear orderings $<$ on
$\m A$.

\end{proof}

\section{$\mathcal{P}(3)$}
\label{section:p3}
Let $\mathcal{P}=(P,\prec)$ be the generic partial order. In this
context, a subset $D$ of $\mathcal{P}$ is \emph{dense} if and only if
for every $x,z\in P$ there is $y\in D$ such that $x\prec y\prec z$ (see \cite{key-C13}).
Partition $P$ into three dense subsets $P_{0}$,$P_{1}$,$P_{2}$. Note
that each $P_{i}$ is isomorphic as a partial order to $\mathcal{P}$.
Define $\mathbb{P}$ to be the directed graph with vertex set $P$
so that $(x,y)$ is an edge if and only if $x\prec y$. Our underlying
signature is $\mathcal{L}=\{E\}$ where $E$ is the binary directed
edge relation. We will temporarily consider extra binary relations $T_i^{X}$, for $i\in\{-1,0,1\}$ and a set of pairs $X$.  For a set of pairs $X$, define the relations $T_i ^ {X}$ as follows:
put $(x,y) \in T_{-1} ^ {X}$ if and only if $(x,y)\in X$;
put $(x,y) \in T_{1} ^ {X}$ if and only if $(y,x)\in X$; finally,
put $(x,y) \in T_{0} ^ {X}$, otherwise.
We define $\mathcal{P}(3)$ as a variant of $\mathbb{P}$ by
"twisting" the directed edges $E ^ {\mathbb P}$: The
set of vertices of $\mathcal{P}(3)$ is $P$. Define the edges $E^{\mathcal{P}(3)}$ so that $(x,y) \in T_{k+(j-i) \mod 3} ^ {E^{\mathcal{P}(3)}}$ if and only if $(x,y) \in T_k ^ {E^{\mathbb P}}$, whenever $x\in P_i$ and $y\in P_j$.

An important observation, which we will
use to show the expansion property, is that a cycle of length three
embeds in $\mathcal{P}(3)$, but not in $\mathcal{P}$.

Now, we define the expansion of $\mathcal{P}(3)$. Let $\mathcal{L}'=\{E,R_{0},R_{1},R_{2},<\}$
be the signature, where $E$ and $<$ are binary and $R_{i}$'s are
unary. We set $E ^ {\mathcal{P}(3) ^ {*}}$ to be $E ^ {\mathcal{P}(3)}$. {Define $\mathcal{P}(3)^{*}$ to be the expansion of $\mathcal{P}(3)$
where $R_{i}^{\mathcal{P}(3)^{*}}=P_{i}$ and $<^{\mathcal{P}(3)^{*}}$
is a linear extension of $\prec$ such that $\mathcal{P}(3)^{*}$ is still a Fra\"{i}ss\'{e} structure. Define
$\mathcal{K}$ to be the age of $\mathcal{P}(3)^{*}$. 

Next, we define signatures $\mathcal{L}_{0}=\{\prec,R_{0},R_{1},R_{2},<\}$,
$\mathcal{L}_{1}=\{\prec,<\}$, and $\mathcal{L}_{2}=\{R_{0},R_{1},R_{2},<\}$,
where $R_{i}$'s are unary, while $\prec$ and $<$ are binary.

Define $\mathcal{K}_{0}$ to be the class of finite structures $(A,\prec^{\mathbf{A}},R_{0}^{\mathbf{A}},R_{1}^{\mathbf{A}},R_{2}^{\mathbf{A}},<^{\mathbf{A}})$
where $\prec^{\mathbf{A}}$ is a poset on $A$, $R_{i}^{\mathbf{A}}$'s
partition $A$ and $<^{\mathbf{A}}$ is a linear order extending $\prec^{\mathbf{A}}$.
Define $\mathcal{K}_{i}$ to be the class of reducts of structures
from $\mathcal{K}_{0}$ with signature $\mathcal{L}_{i}$, for $i\in\{1,2\}$.  The class $\mathcal{K}_{1}$ already satisfies the Ramsey property, see \cite{key-PTW}.  Moreover, the class $\mathcal{K}_{2}$ is essentially the previously defined class $\mathcal{OP}_{3}$ and its Ramsey property was demonstrated in \cite{key-KPT} (see proof of Theorem 8.4).

The members of the class $\mathcal{K}_{0}$ can be treated as members of the class $\mathcal{K}$ and vice versa. For simplicity, we suppose a given $\m A \in \mathcal {K}_{0}$ is such that $(A,\prec^{\m A})$ is a substructure of $\mathcal{P}$.  By the "twisting" procedure defined at the beginning of this section, we arrive at a member $\m A ' = (A,E^{\m A'})$ of the age of $\mathcal{P}(3)$ so that $(A,E^{\m A'},R_0^{\m A},R_1^{\m A},R_2^{\m A},<^{\m A})$ is in the age of $\mathcal{P}(3)^{*}$.  By reversing this procedure, we may turn a given member of $\mathcal{K}$ into a member of $\mathcal{K}_{0}$.

What follows is a short but powerful argument from \cite{key-S3}, which
we include here to keep this paper as self-contained as possible, as before.  The reader is invited to compare it to the proof of the result from \cite{key-B} found in \cite{key-S1}.
\begin{thm}
The class $\mathcal{K}_{0}$ satisfies the Ramsey property.\end{thm}
\begin{proof}
Let \[ \mathbf{A}=(A,E^{\mathbf{A}},R_{0}^{\mathbf{A}},R_{1}^{\mathbf{A}},R_{2}^{\mathbf{A}},<^{\mathbf{A}})\]
and \[\mathbf{B}=(A,E^{\mathbf{B}},R_{0}^{\mathbf{B}},R_{1}^{\mathbf{B}},R_{2}^{\mathbf{B}},<^{\mathbf{B}})\]
be in $\mathcal{K}$. Set $\mathbf{A}_{1}=(A,E^{\mathbf{A}},<^{\mathbf{A}})$,
\[\mathbf{A}_{2}=(A,R_{0}^{\mathbf{A}},R_{1}^{\mathbf{A}},R_{2}^{\mathbf{A}},<^{\mathbf{A}}),\]
\[\mathbf{B}_{1}=(A,E^{\mathbf{B}},<^{\mathbf{B}}),\] \[\mathbf{B}_{2}=(A,R_{0}^{\mathbf{B}},R_{1}^{\mathbf{B}},R_{2}^{\mathbf{B}},<^{\mathbf{B}}).\]

Let $r$ be a natural number. We have $\mathbf{C}_{1}\in\mathcal{K}_{1}$
and $\mathbf{C}_{2}\in\mathcal{K}_{2}$ such that

\[
(\mathbf{C}_{1},\mathbf{C}_{2})\rightarrow(\mathbf{B}_{1},\mathbf{B}_{2})_{r}^{(\mathbf{A}_{1},\mathbf{A}_{2})}
\]

Define $\mathbf{C}\in\mathcal{K}$ as follows: Set the universe $C=C_{1}\times C_{2}$.
Set $(x,y)\prec^{\mathbf{C}}(z,w)$ if and only if $x\prec^{\mathbf{C}_{1}}z$.
Set $(x,y)\in R_{i}^{\mathbf{C}}$ if and only if $y\in R_{i}^{\mathbf{C}_{2}}$.
Finally, set $(x,y)<^{\mathbf{C}}(z,w)$ if and only if $x<^{\mathbf{C}_{1}}z$
or ($x=z$ and $y<^{\mathbf{C}_{2}}w$). For structures $\mathbf{D}_{1}\subseteq\mathbf{C}_{1}$,
$\mathbf{D}_{2}\subseteq\mathbf{C}_{2}$ of the same size, define
the \emph{diagonal} substructure $\Delta(\mathbf{D}_{1},\mathbf{D}_{2})\subseteq\mathbf{C}$
as follows: First define an ordering on $[|D_{1}|]\times[|D_{2}|]$
by setting $(x,y)<(z,w)$ if and only if $x<z$ or ($x=z$ and $y<w$).
There is a unique order preserving map $f:[|D_{1}|]\times[|D_{2}|]\rightarrow D_{1}\times D_{2}$.
Set the universe of $\Delta(\mathbf{D}_{1},\mathbf{D}_{2})$ to be
$\{f(1,1),f(2,2),\ldots,f(|D_{1}|,|D_{2}|)\}$.

Let
\[
\chi:{\mathbf{C} \choose \mathbf{A}}\rightarrow[r].
\]
 be a colouring. Define the induced colouring $\tilde{\chi}:{(\mathbf{C}_{1},\mathbf{C}_{2}) \choose (\mathbf{A}_{1},\mathbf{A}_{2})}\rightarrow[r]$
by

\[
\tilde{\chi}((\mathbf{A}'_{1},\mathbf{A}'_{2}))=\chi(\Delta(\mathbf{A}'_{1},\mathbf{A}'_{2})).
\]

Then there is a $\tilde{\chi}$-monochromatic \[ (\mathbf{B}'_{1},\mathbf{B}'_{2})\in{(\mathbf{C}_{1},\mathbf{C}_{2}) \choose (\mathbf{B}_{1},\mathbf{B}_{2})}.\]
Set $\mathbf{B}'=\Delta(\mathbf{B}'_{1},\mathbf{B}'_{2})$. Then $\mathbf{B}'$
is isomorphic to $\mathbf{B}$ and is $\chi$-monochromatic.\end{proof}
\begin{cor}
The age of $\mathcal{P}(3)^{*}$ satisfies the Ramsey property.\end{cor}
\begin{thm}
The age of $\mathcal{P}(3)^{*}$ satisfies the expansion property
with respect to the age of $\mathcal{P}(3)$.\end{thm}
\begin{proof}
We first define a list of elements of the age of $\mathcal{P}(3)^{*}$:
\begin{itemize}
\item If $i\in\{0,1,2\}$, set $\mathbf{X}_{i}$ to be such that $X_{i}=\{x\}$,
$E^{\mathbf{X}_{i}}=<^{\mathbf{X}_{i}}=\emptyset$. Moreover, set
$R_{i}^{\mathbf{X}_{i}}=\{x\}$ and $R_{i}^{\mathbf{X}_{j}}=\emptyset$
for $i\neq j$.
\item If $i\in\{0,1,2\},$ set $\mathbf{D}_{i}$ to be such that $D_{i}=\{x,y,z\}$,\\
 $E^{\mathbf{D}_{i}}=\{(x,y),(y,z),(z,x)\}$, $R_{i}^{\mathbf{D}_{i}}=\{x,y\}$,
$R_{(i+1)\mod3}^{\mathbf{D}_{i}}=\{z\}$, $<^{\mathbf{D}_{i}}=\{(x,y),(y,z),(x,z)\}$.
\item If $i,j\in\{0,1,2\}$, $i\leq j$, set $\mathbf{E}_{i,j}^{+}$ to
be such that $E_{i,j}^{+}=\{x,y\}$, $x\in R_{i}^{\mathbf{E}_{i,j}^{+}}$,
$y\in R_{j}^{\mathbf{E}_{i,j}^{+}}$. Set $<^{\mathbf{E}_{i,j}^{+}}=\{(x,y)\}$.
If $i=j$, set $E^{\mathbf{E}_{i,j}^{+}}=\emptyset$. If $j=i+1$,
set $E^{\mathbf{E}_{i,j}^{+}}=\{(y,x)\}$. If $j=i+2$, set $E^{\mathbf{E}_{i,j}^{+}}=\{(x,y)\}$.
\item If $i,j\in\{0,1,2\}$, $i<j$, set $\mathbf{E}_{i,j}^{-}$ to be such
that $E_{i,j}^{-}=\{x,y\}$, $x\in R_{i}^{\mathbf{E}_{i,j}^{-}}$,
$y\in R_{j}^{\mathbf{E}_{i,j}^{-}}$. Set $<^{\mathbf{E}_{i,j}^{-}}=\{(y,x)\}$.
If $j=i+1$, set $E^{\mathbf{E}_{i,j}^{-}}=\{(y,x)\}$. If $j=i+2$,
set $E^{\mathbf{E}_{i,j}^{-}}=\{(x,y)\}$.

\item For $i\in\{0,1,2\}$, set $\mathbf{F}_{i}$ to be such that $F_{i}=\{x,y,z\}$.
Set $<^{\mathbf{F}_{i}}=\{(x,y),(y,z),(x,z)\}$. Set $E^{\mathbf{F}_{i}}=\{(x,z)\}$,
$R_{i}^{\mathbf{F}_{i}}=\{x,y,z\}$, $R_{k}^{\mathbf{F}_{i}}=\emptyset$
for $k\neq i$.

\item For $i,j\in\{0,1,2\},$ $i\neq j$, set $\mathbf{G}_{i,j}$ to be
such that $G_{i,j}=\{x,y,z\}$, $R_{i}^{\mathbf{G}_{i,j}}=\{x,y\}$,
$R_{j}^{\mathbf{G}_{i,j}}=\{z\}$, $<^{\mathbf{G}_{i,j}}=\{(z,x),(x,y),(z,y)\}$.
If $j=i+1$, then set $E^{\mathbf{G}_{i,j}}=\{(x,z),(y,z)\}$. If
$j=i+2$, then set $E^{\mathbf{G}_{i,j}}=\{(x,y),(y,z)\}$. If $j=i-1$,
then set $E^{\mathbf{G}_{i,j}}=\{(y,z)\}$. If $j=i-2$, then set
$E^{\mathbf{G}_{i,j}}=\{(x,z),(z,y)\}$.

\end{itemize}
Note that $\mathbf{D}_{i}$'s are cycles.  Let $\mathbf{A}$ be in the age of
$\mathbf{\mathcal{P}}(3)^{*}$. Let $\mathbf{A}'=(A,E^{\mathbf{A}})$
be its reduct. List all possible expansions of $\mathbf{A}'$ that
result in elements of the age of $\mathcal{P}(3)^{*}$: $\mathbf{A}_{1},\ldots\mathbf{A}_{k}$.
Let $\mathbf{B}$ be a structure in the age of $\mathcal{P}(3)^{*}$
such that $\mathbf{A}$, $\mathbf{A}_{i}$'s, $\mathbf{X}_{i}$'s,
$\mathbf{E}_{i,j}^{+/-}$'s,  $\mathbf{F}_{i}$'s and $\mathbf{G}_{i,j}$'s
all embed in $\mathbf{B}$. We can find such a $\mathbf{B}$ due to
the joint embedding property. Using the Ramsey property, we can find
$\mathbf{C}_{i}$ in the age of $\mathcal{P}(3)^{*}$ such that
\[
\mathbf{C}_{0}\rightarrow(\mathbf{B})_{3}^{\mathbf{X}_{0}}
\]

\[
\mathbf{C}_{1}\rightarrow(\mathbf{C}_{0})_{3}^{\mathbf{X}_{1}}
\]

\[
\mathbf{C}_{2}\rightarrow(\mathbf{C}_{1})_{3}^{\mathbf{X}_{2}}.
\]

Fix a bijection $g:\{0,\ldots,5\}\rightarrow(\{0,1,2\}\times\{0,1,2\})\backslash\{(i,j):i>j\}$.
Find

\[
\mathbf{C}_{3+i}\rightarrow(\mathbf{C}_{2+i})_{2}^{\mathbf{E}_{g(i)}^{+}}
\]
 for $0\leq i\leq5$.

Fix a bijection $h:\{0,\ldots,2\}\rightarrow(\{0,1,2\}\times\{0,1,2\})\backslash\{(i,j):i\geq j\}$.
Find

\[
\mathbf{C}_{9+i}\rightarrow(\mathbf{C}_{8+i})_{2}^{\mathbf{E}_{h(i)}^{-}}
\]
 for $0\leq i\le2$.

Set $\mathbf{C}$ to be $\mathbf{C}_{11}$. Let $\mathbf{C}'=(C,E^{\mathbf{C}})$
be a reduct of $\mathbf{C}$. Consider any expansion $\mathbf{C}^{*}$
of $\mathbf{C}'$ in the age of $\mathcal{P}(3)^{*}$. For $i\in\{0,1\}$,
define colourings $\chi_{i}:{\mathbf{C} \choose \mathbf{X}_{i}}\rightarrow\{0,1,2\}$
so that $\chi(\mathbf{X}')=i$ where $X'=\{x\}$ and $x\in R_{i}^{\mathbf{C}^{*}}$.
For $0\leq i\leq5$, define $\chi_{3+i}:{\mathbf{C} \choose \mathbf{E}_{g(i)}^{+}}\rightarrow\{0,1\}$
by setting $\chi_{3+i}(\mathbf{E}')=0$ whenever $\mathbf{E}'=\{x,y\}$,
$x<^{\mathbf{C}}y$ and $x<^{\mathbf{C}^{*}}y$. Also, set $\chi_{3+i}(\mathbf{E}')=1$
whenever $\mathbf{E}'=\{x,y\}$, $x<^{\mathbf{C}}y$ but $x>^{\mathbf{C}^{*}}y$.
Likewise, for $0\leq i\leq3$, define $\chi_{9+i}:{\mathbf{C} \choose \mathbf{E}_{g(i)}^{-}}\rightarrow\{0,1\}$
by setting $\chi_{9+i}(\mathbf{E}')=0$ whenever $\mathbf{E}'=\{x,y\}$,
$x<^{\mathbf{C}}y$ and $x<^{\mathbf{C}^{*}}y$. Furthermore, set
$\chi_{9+i}(\mathbf{E}')=1$ whenever $\mathbf{E}'=\{x,y\}$, $x<^{\mathbf{C}}y$
but $x>^{\mathbf{C}^{*}}y$. We can find $\mathbf{B}'\in{\mathbf{C} \choose \mathbf{B}}$
that is $\chi_{i}$-monochromatic, for all $i$. Set $k_{i}$ to be
the colour assigned by $\chi_{i}$ for elements in ${\mathbf{B}' \choose \mathbf{X}_{i}}$,
for $i<3$.

We first show that the parts of $\mathbf{B}'$ determined by $R_{i}$'s
remain the same in $\mathbf{C}^{*}$ (other than being perhaps reordered):
more specifically, we show that $k_{i}\neq k_{j}$, for $i\neq j$,
$i,j<3$. Suppose otherwise, i.e., $k_{i_{0}}=k_{j_{0}}$, for some
$i_{0}\neq j_{0}$, $i_{0},j_{0}<3$. Consider a cycle $\mathbf{D}\in{\mathbf{B}' \choose \mathbf{D}_{i_{0}}}$.
For all $x\in D$, we have $x\in R_{i_{0}}^{\mathbf{D}}$. This means
that $\mathbf{D}$ embeds in $P_{i_{0}}$ which is impossible because
$(P_{i_{0}},\prec)$ is isomorphic to the generic partial order.

We can show that for every $i>3$, the colouring $\chi_{i}$ cannot
assign the colour $1$ on ${\mathbf{B}' \choose \mathbf{E}_{g(i)}^{+}}$
or ${\mathbf{B}' \choose \mathbf{E}_{h(i)}^{-}}$: Suppose otherwise.
Let $k\in\{0,1,2\}$. Then $\mathbf{F}_{k}$ embeds in $\mathbf{B}'$. If the colour
of every copy of $\mathbf{E}_{k,k}^{+}$ in $\mathbf{B}'$ is assumed
to be $1$, we have that $<^{\mathbf{C}^{*}}$ restricted to $F_{k}\times F_{k}$
is a cycle. This is absurd. This means that every ``non-edge'' in
$\mathbf{B}'$ entirely contained in a part has the same linear extension
with respect to $\mathbf{C}^{*}$ and $\mathbf{C}$. ({*})

Let $k<\ell$, $k,\ell\in\{0,1,2\}$. Then $\mathbf{G}_{k,\ell}$
embeds in $\mathbf{B}'$, as well. By ({*}), the single non-edge of
every copy of $\mathbf{G}_{k,\ell}$ contained entirely in a part
has only one possible linear extension. If the colour of every
copy of $\mathbf{E}_{k,\ell}^{+}$ in $\mathbf{B}'$ is assumed to
be $1$, we have that $<^{\mathbf{C}^{*}}$restricted to $G_{k,\ell}\times G_{k,\ell}$
would again be a cycle which is again absurd. The case when $k>\ell$
is similar, except we replace $\mathbf{E}_{k,\ell}^{+}$ with $\mathbf{E}_{k,\ell}^{-}$.

In other words, the extension $<^{\mathbf{C}^{*}}$ agrees with $<^{\mathbf{C}}$
when both are restricted to $\mathbf{B}'$. Thus, by construction
of $\mathbf{B}$ and since $\mathbf{B}\cong\mathbf{B}'$, any expansion
of $\mathbf{A}^{*}$ embeds in $\mathbf{B}'$ and therefore in $\mathbf{C}^{*}$.\end{proof}

\section{Semigeneric}

\label{section:semi}

The \emph{semigeneric directed graph} is the Fra\"iss\'e limit of the
class $\mathcal S$ made of all those elements of
$\mathrm{Age}(\omega*I_{\omega})$ which satisfy the following parity
constraint: for any two distinct parts $P$ and $P'$, distinct $u,v$ in
$P$ and distinct $x,y$ in $P'$, the number of edges directed from
$u,v$ to $x,y$ is even.

We construct an expansion $\mathcal S ^{*}$ of $\mathcal S$ by adding
two binary relations symbols, $R$ and $<$. As usual, the symbol $<$
will be interpreted as a linear order. Let $\mathbf A \in \mathcal S$
be $k$-partite. Before explaining how an expansion of $\m A$ is built,
let us point out a crucial fact: given two different parts $P$ and
$P'$ of $\mathbf A$, parity constraints imply that the following
relation is well-defined, and is an equivalence relation on $P$:
\[ u\sim^{\mathbf A} _{P'} v \quad
\mathrm{iff} \quad \forall x \in P' \ E^{\mathbf A}(x,u)
\leftrightarrow E^{\mathbf A}(x,v).\]
In fact there are two equivalence classes, one is $\{ u \in P:
E^{\mathbf A}(x,u)\}$ and the other is $\{ v \in P: E^{\mathbf
A}(v,x)\}$, and observe these two classes do not depend on the choice
of $x \in P'$. \\ 
Moreover the relation 
$\sim^{\mathbf A} $ defined on $A$ by:
\[u\sim^{\mathbf A} v \quad \mathrm{ iff }  
\quad \exists \ P \quad \{u,v\} \subseteq P  
\quad \mathrm{ and } \quad \forall P'\neq P \ u\sim ^{\mathbf A} _{P'} v  \]
is also an equivalence relation. Note that the number of
$\sim^{\mathbf A}$-classes is at most $k2^{k-1}$ (at most $2^{k-1}$ classes on
each part), regardless of the actual size of $\mathbf A$, but will
eventually reach this maximum size by the joint embedding property.

We now explain how an expansion $\m A^{*}$ is constructed. We first
consider a linear order $T = \{ t_{i}:i<k\}$, viewed as a directed
graph, where there is an edge from $t_{i}$ to $t_{j}$ whenever
$t_{i}<t_{j}$. This is an element of $\mathcal S$, and thus by the
joint embedding property we can take a $k$-partite element $\bar{\m
A}$ of $\mathcal S$ whose universe is the disjoint union of $A$ and
$T$ and into which both $\m A$ and $\m T$ embed, and thus $T$ becomes a
transversal in $\bar{\m A}$. Next, for $u$ and $x$ in different parts
of $\m A$, we define
\[R^{\mathbf A ^{*}}(x,u) \quad \mathrm{iff} \quad E^{\bar{\m A}}(t_{i},u) \] 
where $t_{i}$ is the only element of $T$ lying in the same part of
$\bar{\mathbf A}$ as $x$. Finally, consider a linear order on $\bar{\m
A}$ that extends the already existing linear order on the copy of $T$
and is convex relative to the $k$-partition of $\bar{\m A}$ (but not
necessarily with respect to $\sim^{\bar{\mathbf A}}$, which is
finer). Then finally let  $<^{\mathbf A ^{*}}$ denote its restriction on $\m
A$. 


Note that adding $R^{\mathbf A ^{*}}$ and $<^{\mathbf A ^{*}}$ as
above corresponds essentially to adding unary predicates $P^{\mathbf
A^{*}}_{i,f}$, with $i<k$ and $f\in [2]^{[k]\smallsetminus \{i\}}$, 
that partition $\mathbf A$ into the $\sim^{\mathbf A} $ equivalence
classes. Indeed, $<^{\mathbf A ^{*}}$ labels the parts of $\mathbf A$
as $P_{0}<^{\mathbf A^{*}} ... <^{\mathbf A ^{*}}P_{k-1}$, and 
$R^{\mathbf A ^{*}}$ gives rise to:
\[ P^{\mathbf A ^{*}}_{i,f} = \{ u\in P_{i} :
\forall j \neq i \ (f(j)=0 \leftrightarrow  \forall x \in P_j \  R^{\mathbf A ^{*}}(x,u))\}.\] 
Note also that the edge relation of $\bar{\mathbf A}$ can be recovered
from $R^{\mathbf A ^{*}}$ and $<^{\mathbf A ^{*}}$, and hence from
$(P^{\mathbf A ^{*}} _{i, f})_{i, f}$ and $<^{\mathbf A ^{*}}$ as
well. To see this, let $x, u\in \bar{A}$ be in different parts. If
both of them belong to the transversal that was used to define
$\bar{\mathbf A}$, then the edge relation is given by the order
relation $<^{\mathbf A ^{*}}$ between any two elements of $\mathbf A
^{*}$ from the respective parts. If exactly one of them belongs to the
transversal set, say $x=t_i$, then knowing whether there is an edge
from $x=t_i$ to $u$ is directly available from the relation $R^{\mathbf A
^{*}}$ evaluated at $(y,u)$ for any $y \in A$ in the same part as
$x=t_i$. On the other hand, if neither $x$ nor $u$ are in the
transversal set, then consider the quadruple formed by $x$, $u$, and
the corresponding elements $t_{x}$ and $t_{u}$ from the transversal
set; in that quadruple, all edges but the one between $x$ and $u$ are
available from $R^{\mathbf A ^{*}}$ and $<^{\mathbf A ^{*}}$ as
described above, and so the orientation of the edge between $x$ and
$u$ can be deduced thanks to the parity constraint.

One last comment: any element of the class of $k$-sequences of members of $\mathcal{OP}_{2^{k-1}}$ whose unary
predicates are indexed by the set of all pairs $(i,f)$ with $i<k$ and
$f\in [2]^{[k]\smallsetminus \{i\}}$ gives raise to a $k$-partite
element of $\mathcal S ^{*}$. Let us illustrate this on an example
that will also be useful later on. Consider the set $S[n,k]$ defined
as the collection of triples $(i,f,j)$ where $i<k$, $f\in
[2]^{[k]\smallsetminus\{i\}}$ and $j<n$. For $f$ and $i$ fixed, define
\[P ^{\m S[n,k] ^{*}} _{i,f} = \{ (i,f,j) : j<n\}.\]
The element of $\mathcal S ^{*}$ we are about to construct will be
$k$-partite with parts: 
\[P ^{\m S[n,k] ^{*}} _{i} = \bigcup_{f\in
[2]^{[k]\smallsetminus\{ i\}}} P ^{\m S[n,k] ^{*}} _{i,f}= \{ (i,f,j)
: f\in [2]^{[k]\smallsetminus\{i\}}, j<n\}.\] 
If, for each $i$, we equip $[2]^{[k]\smallsetminus\{i\}}$ with the
lexicographical order, then we can induce on $S[n,k]$ a natural
lexicographical order $<^{\m S[n,k] ^{*}}$. This is clearly convex. \\
To define $R^{\m S [n,k]^{*}}$, set:
\[R^{\m S[n,k]^{*}}((i',f',j'),(i,f,j)) \quad \mathrm{ iff } \quad i\neq i' \ \mathrm{ and } \  f(i')=0.\]
For the edge relation $E^{\m S [n,k]^{*}}$, we introduce an imaginary transversal 
$\{t_{i}:i<k\}$, where each imaginary point $t_{i'}$ is attached to
the part $P ^{\m S[n,k] ^{*}} _{i'}$. First, add an edge from $t_{i'}$ to $t_{i}$
whenever $i'<i$. Then for each $(i,f,j)$ with $i\neq i'$, add an edge from
$t_{i'}$ to $(i,f,j)$ iff $f(i')=0$, i.e. whenever $R^{\m
S[n,k]^{*}}((i',f',j'),(i,f,j))$ for every point of the form
$(i',f',j')$. Otherwise, add an edge in the other direction. Next,
when $(i,f,j)$ and $(i',f',j')$ are such that $i\neq i'$, consider the
quadruple formed by $(i,f,j)$, $(i',f',j')$, $t_{i}$ and
$t_{i'}$. There is only one way to add an oriented edge between
$(i,f,j)$ and $(i',f',j')$ in such a way that the number of edges from
the pair $\{(i,f,j), t_{i}\}$ to the pair $\{(i',f',j'),t_{i'}\}$ is
even. \\
This is enough to ensure that the resulting directed graph
structure $E^{\m S[n,k] ^{*}}$ is in $\mathcal S$: indeed, consider a
quadruple formed by pairs $\{u, v\}$, $\{ x,y\}$ of elements in same
parts. Add the corresponding imaginary points $t$ and $t'$. For each
transversal pair of $\{u,v,x,y\}$, complete it into a quadruple
using $\{t,t'\}$ where the parity constraint is satisfied. Identifying
$E^{\m S[n,k] ^{*}}$ with its characteristic function, this is
equivalent to saying some sum equals $0$ modulo $2$. For example, for $\{
u,x\}$: 
\[ E^{\m S[n,k] ^{*}}(u,x)+E^{\m S[n,k] ^{*}}(u,t')+E^{\m
S[n,k] ^{*}}(t,x)+E^{\m S[n,k] ^{*}}(t,t')=0.\] 
Adding up all these equalities for all transversal pairs of
$\{u,v,x,y\}$ (i.e. pairs with elements in
different parts), notice that each term involving $t$ or $t'$ appears
an even number of times, and therefore vanishes in the total sum. As a
result, we obtain: 
\[ E^{\m S[n,k] ^{*}}(u,x)+E^{\m S[n,k] ^{*}}(u,y)+E^{\m S[n,k] ^{*}}(v,x)+E^{\m S[n,k] ^{*}}(v,y)=0.\]
This shows that the parity constraint is satisfied on $\{u,v,x, y\}$,
and shows that $\m S[n,k]^{*}\in \mathcal S^{*}$. We call $\m S[n,k]$
the corresponding reduct in $\mathcal S$. Let us emphasize the following
property of $\m S[n,k]$:

\begin{lem} 
\label{lem:Snk}
Every part of $\m S[n,k]$ is partitioned into $2^{k-1}$ many $\sim^{\m
S[n,k]}$-classes, each of size $n$.
\end{lem}

\begin{proof}
Fix $i<k$. By construction, the $\sim^{\m S[n,k]}$-classes inside the
$i$-th part of $\m S[n,k]$ are the sets $P ^{\m S[n,k] ^{*}} _{i,f}=
\{ (i,f,j) : j<n\}$, where $f$ runs over
$[2]^{[k]\smallsetminus\{i\}}$. Clearly, each has size $n$.
\end{proof}

We define $\mathcal AS$ to consist of the elements of $\mathcal S^*$ with the 
ordering removed.  Remembering how an expansion in $ \mathcal S^*$ is intuitively built
(by adding a transversal), a direct consequence of the previous lemma
is the following corollary:

\begin{cor}\label{cor:Snk}
All expansions of $\m S[n,k]$ in $\mathcal AS$ are isomorphic. 
\end{cor}

We will use this result later on to prove the expansion property of
$\mathcal S^{*}$ relative to $\mathcal S$. Before that, here is
another direct consequence of the various facts we stated above,
together with the Ramsey property for $k$-sequences of members $\mathcal{OP}_{2^{k-1}}$ (see Lemma \ref{lem:blob2}):

\begin{cor}
\label{cor:semigen2}Suppose $\mathbf{A}^{*}, \mathbf{B}^{*} \in\mathcal{S ^{*}}$
are $k$-partite. Then for any $r$ there exists a $k$-partite $\mathbf{C}^{*} \in\mathcal{S}^{*}$
such that,
\[
\mathbf{C}^{*}\rightarrow(\mathbf{B}^{*})_{r}^{\mathbf{A}^{*}}.
\]
\end{cor}

With this fact in mind, we now prove that $\mathcal S^{*}$ has the
Ramsey property.

\begin{cor}
Suppose $\mathbf A ^*,\mathbf{B}^{*} \in\mathcal{S^{*}}$ are
$k$-partite and $\ell$-partite respectively. Then for any $r$ there
exists an $\ell$-partite $\mathbf{C}^{*}\in\mathcal{S}^{*}$ such that
any $r$-colouring $\chi$ of the copies of $\mathbf A ^*$ with a fixed
trace $\tau$, there exists a $\chi$-monochromatic
$\tilde{\mathbf{B}}\in{\mathbf C ^{*} \choose \mathbf{B}^{*}}$. \\
 We symbolize this by
\[
\mathbf C ^{*}\rightarrow(\mathbf{B}^{*})_{r}^{\mathbf A ^*,\tau}.
\]
\end{cor}

\begin{proof}
Let $\mathbf{D}^{*}$ be the graph induced by the the elements
of $\mathbf{B}^{*}$ with trace $\tau$. Use Corollary \ref{cor:semigen2},
to get a $k$-partite $\mathbf{E}^{*}$ such that
\[
\mathbf{E}^{*}\rightarrow(\mathbf{D}^{*})_{r}^{\mathbf A ^*}.
\]
Now extend each copy of $\mathbf{D}^{*}$ in $\mathbf{E}^{*}$ to a copy
of $\mathbf{B}^{*}$ with trace $[\ell]$, via amalgamation.  Finally, add edges outside
of these copies, as necessary.
\end{proof}

\begin{cor}
\label{cor:semigen4}Suppose $\mathbf A ^*, \mathbf{B}^{*}\in\mathcal{S}^{*}$
are $k$-partite and $\ell$-partite respectively. Then for any $r$ there exists
an $\ell$-partite $\mathbf{C}^{*}\in\mathcal{S}^{*}$ such that
for any $r$-colouring $\chi$ of the copies of $\mathbf A ^*$,
there exists $\tilde{\mathbf{B}}\in{\mathbf C ^{*} \choose \mathbf{B}^{*}}$
such that the colouring $\chi$ restricted to ${\tilde{\mathbf{B}} \choose \mathbf A ^*}$
depends only on the trace.  We symbolize this as
\[
\mathbf C ^{*}\rightarrow(\mathbf{B}^{*})_{r}^{\mathbf A ^*,\mathrm{tr}}.
\]
\end{cor}
\begin{proof}
Enumerate subsets of $[\ell]$ of size $k$: $\tau_{1}$,
$\ldots$, $\tau_{n}$. Set $\mathbf C ^{*}_{1}\rightarrow(\mathbf{B}^{*})_{r}^{\mathbf A ^*,\tau_{1}}$.
Set $\mathbf C ^{*}_{i}\rightarrow(\mathbf C ^{*}_{i-1})^{\mathbf{A}^{*},\tau_{i}}$
for $i>1$. Finally, set $\mathbf C ^{*}=\mathbf C ^{*}_{n}$.
\end{proof}

We are now ready for the full Ramsey result for $\mathcal{S}^{*}$. 

\begin{thm}
The age of $\mathcal{S}^{*}$ is Ramsey.
\end{thm}

\begin{proof}
Suppose $\mathbf A ^*,\mathbf{B}^{*}\in\mathcal{S}^{*}$, and $r$ are
given. We show that there exists an $\tilde{C}^{*}\in\mathcal{S}^{*}$
such that
\[
\mathcal{\mathbf C ^{*}}\rightarrow(\mathbf{B}^{*})_{r}^{\mathbf A ^*}.
\]
First we apply Ramsey's theorem to the traces and use Corollary
\ref{cor:semigen4}. Let $k\leq \ell \in\mathbb{N}$ such that $\m
B^{*}$ is $\ell$-partite and $\m A^{*}$ is a $k$-partite substructure
of $\m B^{*}$. Using Ramsey's theorem, we can find $m$ such that
$m\rightarrow(\ell)^{k}_{r}$.  Enumerate the $\ell$-subsets of $[m]$
as $\tau_1, \ldots, \tau_\ell$. For each $i\leq \ell$ consider a copy
$\m B ^{*} _{i}$ of $\m B$ with trace $\tau_{i}$ in the Fra\"iss\'e
limit $\mathbf S ^{*}$ of $\mathcal S^{*}$ so that the $\m B^{*}
_{i}$'s are pairwise disjoint.  Set $\mathbf{\bar{B}}^{*}$ to be the
$m$-partite substructure of $\mathbf S^{*}$ supported by the union of
these $\m B_{i} ^{*}$'s.

Using corollary \ref{cor:semigen4}, we can find an $m$-partite $\m
C^{*}$ in $\mathbf S^{*}$ such that every $r$-colouring $\chi$ of ${\m
C^{*} \choose \m A^{*}}$, there exists a copy $\mathbf{\bar{B}}' \in
{\m C \choose
\mathbf{\bar{B}}^{*}}$ such that $\chi$ restricted to ${\mathbf{\bar{B}}' \choose \m A^{*}}$
depends only on the trace.

Define $\bar{\chi} : {m \choose k}
\rightarrow [r]$ by setting $\bar{\chi}(\sigma)$ to be the colour of any
copy of $\m A^{*}$ in $\mathbf{\bar{B}}'$ with trace $\sigma$.  This is well defined
by construction.  By choice of $m$, there is a
$\bar{\chi}$-homogeneous $\tau_{i_0}$ of size $\ell$.  By construction, there is
$\m B'\in{ \mathbf{\bar{B}}' \choose \m B^{*}}$ with trace $\tau_{i_0}$, which completes
the proof. 
\end{proof}

It remains to show that $\mathcal S^{*}$ has the expansion property
relative to $\mathcal S$. In order to do this, fix $\m A \in \mathcal
S$ as well as an expansion $\m A^{*}$ of $\m A$ in $\mathcal
S^{*}$. We will be done once the following two lemmas are proved:

\begin{lem}
There exists $\bar{\m B} \in \mathcal{AS}$ for which any expansion of $\bar{\m B}$ in $\mathcal S^{*}$ contains $\m A^{*}$. 
\end{lem}

\begin{proof}
Consider all the expansions $\m A_{0} ^{*},...,\m A_{s} ^{*}$ of $\m
A$ obtained from $\m A^{*}$ by reversing $<^{\m A^{*}}$ on some
parts. Let $\hat{\m A }^{*}$ be $k$-partite and containing them
all. Then $\m A^{*}$ embeds in $\hat{\m A}^{*}$, even when $<^{\hat{\m
A}^{*}}$ is reversed in any of its parts. Next, use the Ramsey
property in $\mathcal S ^{*}$ with colourings of edges and non-edges
as in Theorem \ref{thm:exp7} to create $\bar{\m B}^{*}$ so that any
convex linear ordering $\prec$ on $\bar{\m B}^{*}$ is so that
$(\bar{\m B}^{*}, E^{\bar{\m B}^{*}}, R^{\bar{\m B}^{*}}, \prec)$
contains a copy of $\hat{\m A}^{*}$ up to a reversing of $\prec$ in
some of the parts. Then, the reduct $\bar{\m B}$ of $ \bar{\m B}^{*}$
in $\mathcal{AS}$ is as required.
\end{proof}

\begin{lem}
There exists $\m B \in \mathcal S$ for which any expansion in
$\mathcal{AS}$ contains $\bar{\m B}$.
\end{lem}

\begin{proof}
Observe that for $k$ and $n$ large enough, we can embed the reduct of $\bar {\m B}$ in $\m
S[n,k]$. Then $\m S[n,k]$ has the required property thanks to
Corollary \ref{cor:Snk}.
\end{proof}

\newpage
\section{Conclusion} \label{section:conclusion}

Having constructed Ramsey precompact expansions of all homogeneous
directed graphs and having verified the expansion property in each
case, we are automatically given a list of the respective universal
minimal flows.  Suppose $\mathbf{F}$ is one of the homogeneous
directed graphs and
$\textbf{F}^{*}=(\mathbf{F},\overrightarrow{R^{*}})$ is its precompact
expansion which we have constructed.  The age of $\textbf{F}^{*}$
satisfies the Ramsey property as well as the respective expansion
property and consists of rigid elements.  Consequently, its
automorphism group is extremely amenable.  Let $G$ be the automorphism
group of $\mathbf{F}$.  Moreover, recall that the logic action of $\m
G$ on $\overline{\mathbf{G}\cdot
\overrightarrow{R^{*}}}$ is the universal minimal flow (see Theorem
\ref{thm:UMF}).  In particular, we arrive at the following table,
summarizing the results of this paper, where we omit the cases when $\m T$
is finite (in these cases  refer to the the definitions preceding Theorem \ref{FinTOmega-exp} and Theorem
\ref{IomegaFinT-exp}).

\smallskip

\begin{tabular}{|c|p{0.45\textwidth}|}
\hline 
Homogeneous directed graph(s) $\mathbf{F}$ & Expanded language  $\overrightarrow{R^{*}}$\tabularnewline
\hline 
\hline 
$\mathbf{T}[I_{n}]$, $I_{n}[\mathbf{T}]$ $\hat{\mathbf{T}}$ &
Relations giving the expanded tournament ${\m T}^{*}$ the Ramsey property and the expansion property;
$R_{1}^{\mathbf{F}^{*}},\ldots,R_{n}^{\mathbf{F}^{*}}$ (unary relations which are transversals in the case of $\mathbf{T}[I_{n}]$,  distinguishing
copies of \textbf{$\mathbf{T}$} in the case of   $I_{n}[\mathbf{T}]$, or the copies of the generic partial
order in the case of $\mathcal{P}(3)$), and $<^{*}$ (convex linear ordering on \textbf{$\mathbf{F}^{*}$})\tabularnewline
\hline 
$\mathcal{P}(3)$ & $R_{0}^{\mathbf{F}^{*}},R_{1}^{\mathbf{F}^{*}},R_{2}^{\mathbf{F}^{*}}$ (unary relations distinguishing the copies of the generic partial
order), and $<^{*}$ (linear ordering extending the underlying generic partial order)
\tabularnewline
\hline 
$\mathbf{T}[I_{\omega}]$, $I_{\omega}[\mathbf{T}]$ &
Relations giving the expanded tournament ${\m T}^{*}$ the Ramsey property and the expansion property;
$<^{*}$ (convex linear ordering on \textbf{$\mathbf{F}^{*}$})\tabularnewline
\hline 
$n*I_{\omega}$ & $P_{1}^{\mathbf{F}^{*}},\ldots,P_{n}^{\mathbf{F}^{*}}$ (unary relations distinguishing
each part of $n*I_{n}$), and $<^{*}$ (convex linear ordering on \textbf{$\mathbf{F}^{*}$})\tabularnewline
\hline 
$\omega*I_{\omega}$ & $<^{*}$ (convex linear ordering on \textbf{$\mathbf{F}^{*}$})\tabularnewline
\hline 
semigeneric & $<^{*}$ (convex linear ordering on \textbf{$\mathbf{F}^{*}$}), and $R^{\m F^*}$ (sub-partitioning binary relation) \tabularnewline
\hline 
\end{tabular}

Note that the relations including the linear orders in the above table
arise from taking the Fra\"{i}ss\'{e} limits of our expanded
classes. Note that the semigeneric case is the only one where a special binary
relation was used (in addition to a linear order). \\

Completing the list of Ramsey precompact expansions of homogeneous
structures and subsequently computing their universal minimal flows is
an ever continuing undertaking.  For example, we have examples of
homogeneous metric spaces whose Ramsey precompact expansions are still
unknown.  In particular from the work of Cherlin in \cite{key-C13},
and related more closely to the subject of this paper, are classes of
\emph{metrically} homogeneous graphs, i.e., ones that are homogeneous
with respect to the graph metric.

\end{document}